\documentclass[11pt]{amsart}
\usepackage{amssymb, amsmath, amsthm, bm, latexsym, enumerate, marginnote,  tikz}

\usepackage{mathrsfs}
\newtheorem{lemma}{Lemma}[section]

\newtheorem{corollary}[lemma]{Corollary}
\newtheorem{definition}[lemma]{Definition}
\newtheorem{conjecture}[lemma]{Conjecture}
\newtheorem{remark}[lemma]{Remark}

\newtheorem{theorem}[lemma]{Theorem}
\theoremstyle{definition}
\numberwithin{equation}{section}

\newcommand{\R}{\mathbb R}

\newcommand{\N}{\mathbb N}
\newcommand{\eps}{\varepsilon}

\newcommand{\supp}{\operatorname{supp}}

\newcommand{\dif}{\mathrm{d}}
\newcommand{\beq}{\begin{equation}}
\newcommand{\eeq}{\end{equation}}
\newcommand{\beqq}{\begin{equation*}}
\newcommand{\eeqq}{\end{equation*}}
\newcommand{\ben}{\begin{eqnarray}}
\newcommand{\een}{\end{eqnarray}}
\newcommand{\beno}{\begin{eqnarray*}}
	\newcommand{\eeno}{\end{eqnarray*}}

\def \f{\frac}
\def\d{\delta}
\def\a{\alpha}
\def\e{\varepsilon}
\def\ld{\lambda}
\def\p{\partial}
\def\v{\varphi}

\theoremstyle{remark}

\numberwithin{equation}{section}

\newcommand{\F}{{\mathcal F}}

\newcommand{\ra}{{\textnormal{RapDec}}}

\begin{document}

	\def\d{\delta}
	\def\a{\alpha}
	\def\e{\varepsilon}
	\def\ld{\lambda}
	\def\p{\partial}
	\def\v{\varphi}
	
	\title[Square function estimates for FIO]{Square function estimates and Local smoothing for Fourier Integral Operators}
	\begin{abstract}
		We prove  a variable coefficient version of the square function estimate of Guth--Wang--Zhang. By a classical argument of Mockenhaupt--Seeger--Sogge, it implies the full range of sharp local smoothing estimates for $2+1$ dimensional Fourier integral operators satisfying the cinematic curvature condition. In particular, the local smoothing conjecture for wave equations on compact Riemannian surfaces is settled.
	\end{abstract}
		\author{Chuanwei Gao}
	\address{School of Mathematical Sciences, Capital Normal University, Beijing 100048, China}
	\email{cwgao@cnu.edu.cn}

	\author{Bochen Liu}
	\address{Department of Mathematics \& International Center for Mathematics, Southern University of Science and Technology, Shenzhen 518055, China}
	\email{Bochen.Liu1989@gmail.com}
	
	\author{Changxing Miao}
	\address{Institute for Applied Physics and Computational Mathematics, Beijing, China}
	\email{miao\textunderscore changxing@iapcm.ac.cn}
	\author{Yakun Xi}
	\address{School of Mathematical Sciences, Zhejiang University, Hangzhou 310027, PR China}
	\email{yakunxi@zju.edu.cn}
	
	\thanks{}
	\maketitle
	\maketitle
	\setcounter{tocdepth}{1}
	\tableofcontents
	\section{Introduction}
	\label{sect:introd}
	\subsection{Local smoothing conjectures}Suppose $(M,g)$ is a $d$-dimensional smooth Riemannian manifold without boundary, and $u(x,t)$ is a solution to the Cauchy problem
	\begin{equation}
	\label{eq:Wave}
	\begin{cases}
	(\partial_t^2-\Delta_g )\,u(x,t)= 0,\,(x,t)\in M\times \R\,\\
	u(0,x)=u_0(x),\ 
	\partial_tu(x,0)=u_1(x)\\
	\end{cases} ,
	\end{equation} where $\Delta_g$ denotes the associated  Beltrami--Laplacian operator. 
	
	When the underlying manifold is Euclidean, i.e., $M=\R^d$ with flat metric, Sogge \cite{Sogge91} conjectured that for $p\ge p_d= \frac{2d}{d-1}$, the space-time norm $\|u\|_{L^p(\R^d\times[1,2])}$ is bounded by suitable Sobolev norms of the initial data, which represents a gain in regularity compared to fixed time estimates.
	\begin{conjecture}
		[{\it Local smoothing in $\R^{d+1}$\rm}]\label{Conj. 1} Suppose $M=\R^d$ equipped with flat metric, $u$ is a solution to the Cauchy problem \eqref{eq:Wave}, and $s_p=\frac{d-1}2-\frac {d-1}p$. Then for any $p\ge p_d=\frac{2d}{d-1}$, $\sigma<\frac1p$, we have
		\beq
		\label{eq:Conj. 1}
		\|u\|_{L^p(\R^d\times[1,2])}\le C_\sigma(\|u_0\|_{L^p_{s_p-\sigma}(\R^d)}+\|u_1\|_{L^p_{s_p-1-\sigma}(\R^d)}).
		\eeq
	\end{conjecture}
	Here $s_p$ represents the order of regularity from the fixed time estimate, thus Conjecture \ref{Conj. 1} asserts that integrating in time provides a $\sigma<1/p$ order of  regularity gain for $p\ge p_d$.
	Conjecture \ref{Conj. 1}, if true, would be a remarkable result that would imply several major open problems in modern harmonic analysis, including the Bochner--Riesz mean, Fourier restriction, and Kakeya--Nikodym conjectures, all of which are wide open in dimensions $d\ge3$. See e.g., \cite{Tao99} for a detailed discussion. In a recent work of Guth--Wang--Zhang \cite{GWZ}, Conjecture \ref{Conj. 1} was verified for $d=2$. 
	
	The situation for a general compact Riemannian manifold $(M,g)$ is quite different. It was shown by Minicozzi and Sogge \cite{MS} (see also \cite{SXX}) that there exist compact manifolds $(M,g)$ which exhibit a certain Kakeya  compression  phenomenon that forbids favorable Kakeya-type estimates. Therefore, for such manifolds, local smoothing fails for all $\sigma<1/p$ if
	\beq
	 p<\bar p_{d,+}:=
	\begin{cases}
		\frac{2(3d+1)}{3d-3},\quad \text{if}\; d\; \text{is odd},\\
		\frac{2(3d+2)}{3d-2}, \quad \text{if}\; d\;\text{is even}.
	\end{cases}
	\eeq
	This result naturally leads to the following conjecture.
	\begin{conjecture}
		[{\it Local smoothing for compact manifolds\rm}]\label{Conj. 2} Let $(M,g)$ be a compact Riemannian manifold with $u$ being a solution to the Cauchy problem \eqref{eq:Wave}. For any $p\ge \bar p_{d,+}$, $\sigma<\frac1p$, we have
		\beq
		\label{eq:Conj. 2}
		\|u\|_{L^p(\R^d\times[1,2])}\le C_\sigma(\|u_0\|_{L^p_{s_p-\sigma}(\R^d)}+\|u_1\|_{L^p_{s_p-1-\sigma}(\R^d)}).
		\eeq
	\end{conjecture}
	Conjecture \ref{Conj. 2} is  open in all dimensions $d\ge2$.
	A classical way of attacking Conjecture \ref{Conj. 2} is to utilize a local parametrix to the Cauchy problem \eqref{eq:Wave} in terms of {\it Fourier integral operators}\footnote{See, e.g., \cite{Horm} for a thorough treatment for the basic theory of FIOs.} (FIOs).  In particular, it is known that
	\begin{equation}
	\label{eq:FIO}
	u(x,t)=\mathcal{ F}_0u_0(x,t)+\mathcal{F}_1u_1(x,t),
	\end{equation}
	where $\mathcal F_j\in I^{j-\frac14}(M\times \R,M;
	\mathscr C)$ are FIOs with canonical relation $\mathscr C$ satisfying the {\it cinematic curvature condition}. 	
	In local coordinates, a Fourier integral operator mapping functions on $\R^d$ to ones on $\R^{d+1}$ can be written as a finite sum of operators of the form
	\beq \label{eq:defFIO}
	\mathcal{F}f(z)=\frac{1}{(2\pi)^d}\int_{\R^d} e^{i\phi(z,\eta)}\,a(z,\eta)\, \hat{f}(\eta)\,d\eta,\quad z=(x,t)\in\R^{d+1},\ \eta\in \R^d,
	\eeq
	where $a\in S^{\sigma}(\R^{d+1}\times\R^d)$ is a symbol with compact support in $(x,t)$, and $\phi\in C^\infty(\R^{d+1}\times (\R^d\setminus \{0\}))$ is a phase function homogeneous of degree 1 in $\eta$. We say that $\mathcal F$ satisfies the (local) cinematic curvature condition if the following conditions are satisfied.
	
	\begin{itemize}
		\item{\bf Non-degeneracy condition.} \beq\label{CC1}{\rm rank} \;\partial_{x\eta}^2 \phi(x,t,\eta)=d\textnormal{ for all } (z,\eta)=(x,t,\eta)\in  {\rm supp}\;a.\eeq
		\item{\bf Cone condition.}  Consider the Gauss map $G: {\rm supp}\;a \rightarrow S^{d}$ by $G(z,\eta):=\frac{G_0(z,\eta)}{|G_0(z,\eta)|}$ where
		\beqq
		G_0(z,\eta):=\bigwedge^d_{j=1}\partial_{\eta_j}\partial_z\phi(z,\eta).
		\eeqq
		The curvature condition
		\beq\label{CC2}
		{\rm rank}\; \partial_{\eta\eta}^2\langle \partial_z\phi(z,\eta), G(z,\eta_0) \rangle|_{\eta=\eta_0}=d-1
		\eeq
		holds for all $(z,\eta_0)\in {\rm supp}\;a$.
	\end{itemize}
	
	Although we introduced the concept of FIOs in relation to Conjecture \ref{Conj. 2}, the numerology for their local smoothing conjectures is yet again different. It is shown in \cite{BHS} that examples constructed by Bourgain \cite{Bourgain91,Bourgain95} can be used to identify FIOs satisfying the cinematic curvature condition for which local smoothing fails for all $\sigma<1/p$, if 
	\beq
	 p<\bar{p}_{d}:=
	\begin{cases}
		\frac{2(d+1)}{d-1},\quad \text{if}\; d\; \text{is odd},\\
		\frac{2(d+2)}{d}, \quad \text{if}\; d\;\text{is even}.
	\end{cases}
	\eeq
	The local smoothing conjecture for FIOs is the following.
	\begin{conjecture}
		[{\it Local smoothing for FIOs\rm}]\label{Conj. 3} Let $\mathcal F$ be a FIO with symbol of order $\mu$ satisfying the cinematic curvature condition. For any $p\ge \bar p_{d}$, $\sigma<\frac1p$, we have
		\beq\label{eq:Conj. 3}
		\|\mathcal F f\|_{L^p(\R^{d+1})}\le C_\sigma\|f\|_{L^p_{\mu+s_p-\sigma}(\R^d)}.
		\eeq
	\end{conjecture}
	We remark that Conjecture \ref{Conj. 2} essentially corresponds to the sub-case of Conjecture \ref{Conj. 3} when the matrix in \eqref{CC2} is positive definite.	Recently, Conjecture \ref{Conj. 3} has been confirmed by Beltran--Hickman--Sogge \cite{BHS} for odd spatial dimensions $d$. Nevertheless, Conjecture \ref{Conj. 3} is open for all even values of $d\ge2$. 	
	\subsection{Main results}
	The purpose of this article is to confirm Conjecture \ref{Conj. 3} for $d=2$.
	\begin{theorem}
		[{\it Local smoothing for FIOs in $2+1$ dimensions\rm}]\label{main 1} Let $\mathcal F$ be a FIO with symbol of order $\mu$ satisfying the cinematic curvature condition. For any $p\ge \bar p_{2}=4$, $\sigma<\frac1p$, we have
		\beq
		\|\mathcal F f\|_{L^p(\R^{2+1})}\le C_\sigma\|f\|_{L^p_{\mu+s_p-\sigma}(\R^2)}.
		\eeq
	\end{theorem}
	We remark that $p_d\le\bar p_{+,d}\le \bar p_d$, while  $p_d=\bar p_{+,d}=\bar p_d$ if and only if $d=2$. Therefore, as a direct consequence of Theorem \ref{main 1}, we also confirm Conjecture \ref{Conj. 2} for $d=2$.
	
	\begin{corollary}
		[{\it Local smoothing for compact surfaces\rm}]\label{main 2} Let $(M,g)$ be a compact Riemannian surface and $u$ be a solution to the Cauchy problem \eqref{eq:Wave}. For any $p\ge \bar p_{2,+}=4$, $\sigma<\frac1p$, we have
		\beq
		\|u\|_{L^p(M\times[1,2])}\le C_\sigma(\|u_0\|_{L^p_{s_p-\sigma}(M)}+\|u_1\|_{L^p_{s_p-1-\sigma}(M)}).
		\eeq
	\end{corollary}
	Let us take a moment to briefly survey the history of this problem. The study of Conjecture \ref{Conj. 1} in $2+1$ dimensions dates back to \cite{Sogge91} where it was first introduced. In the same article, Sogge also established the first local smoothing estimate with a nontrivial gain of regularity at the endpoint $p=4$. This result was improved by Mockenhaupt, Seeger and Sogge \cite{MSSA}, while no estimate with a sharp gain of regularity was established. Further improvements can be found in \cite{TV00, Wolff01, Lee16}. The seminal paper of Wolff \cite{Wolff} introduced an important tool now known as  {\it decoupling inequality} for the cone, and used it to obtain sharp local smoothing estimates for $p\ge 74$. The methods of Wolff were further developed and extended  by several authors (\cite{LW, GS09, GS10, Bourgain13}). In \cite{BD}, Bourgain and Demeter proved a sharp decoupling inequality for the cone for all dimensions $d\ge2$, which naturally leads to local smoothing estimates in $\R^{d+1}$. In particular, when $d=2$, their results imply sharp local smoothing estimates for all $p\ge6$. On the other hand, a certain type of {\it square function estimate} is shown to imply local smoothing results by \cite{MSSA}. Following this route, Lee and Vargas \cite{LV} proved sharp local smoothing estimates for $p=3$ by proving a sharp square function estimate using multilinear restriction estimates.  As mentioned earlier, the recent work of Guth, Wang and Zhang \cite{GWZ} settled Conjecture \ref{Conj. 1} for $d=2$ completely, in which the authors established a sharp $L^4$ square function estimate. 
	
	Conjecture \ref{Conj. 2} and \ref{Conj. 3} were studied in parallel to the developments of Conjecture \ref{Conj. 1} \cite{MSSJ, LS}. 	As mentioned above, Beltran, Hickman and Sogge \cite{BHS} established Conjecture \ref{Conj. 3} for odd spatial dimensions by proving a variable coefficient analog of Bourgain--Demeter decoupling inequality. In their proof, the authors exploited the multiplicativity of the decoupling constants and reduced the problem to proving a decoupling inequality at small scales via an induction on scales argument. At small enough scales, such a decoupling inequality follows directly from its Euclidean counterpart. More recent partial results can be found in \cite{GMY1, GMY2}.
	See the excellent survey article by Beltran, Hickman and Sogge \cite{BHSSurvey} for a detailed discussion of all three conjectures. 
	
	As discussed in \cite{GWZ}, the sharp decoupling inequality for the cone does not imply a full range of sharp local smoothing estimates for Euclidean wave equations. This is one of the reasons why the authors resort to proving sharp square function estimates instead.   Similarly, the variable coefficient decoupling inequalities of Beltran--Hickman--Sogge \cite{BHS} do not fully resolve Conjecture \ref{Conj. 3} in even spatial dimensions.  In this paper, Theorem \ref{main 1} will be proven as a consequence of a variable coefficient generalization of the square function estimates of Guth--Wang--Zhang, which we shall describe in the next subsection. 
	
	\subsection{Square function estimates for FIOs}	The decomposition associated with our square function estimates is classical and the same as that in \cite{MSSJ},  \cite{BHS}. Consider a Fourier integral operator $\mathcal F$ of the form \eqref{eq:defFIO} satisfying the cinematic curvature condition (that is, \eqref{CC1} and \eqref{CC2}) with a symbol of order $0$ and $d=2$. Fix $\lambda\gg 1$. By a standard Littlewood--Paley decomposition,
	one may reduce Theorem \ref{main 1} to proving that for any $\eps>0$,
	\beq \label{eq:34}
	\|\mathcal F^\lambda f\|_{L^4(\R^3)}
	\leq C_\varepsilon \lambda^{\frac14+\varepsilon}\|f\|_{L^4{(\R^2})}, \quad
	\eeq
	where
	$\mathcal F^\lambda $ is defined as
	\begin{equation}
	\label{mklsmcd}
	\mathcal{F}^\lambda f(x,t):=
	\int
	e^{i\phi^\lambda(x,t,\eta)}
	a^\lambda (x,t,\eta)\hat{f}(\eta)\,\dif\eta,
	\end{equation}
	and 
	\beq
	\phi^\lambda(x,t,\eta):=\lambda\phi(x/\lambda,t/\lambda,\eta),\ a^\lambda(x,t,\eta):=  a(x/\lambda,t/\lambda,\eta).
	\eeq
Here  and throughout, we use the convention that $(x,t)$ are the physical variables and $\eta$ are the frequency variables. We may assume that the physical support of $a(x,t,  \eta)\in C_c^\infty(\R^3\times \R^2)$ is small, and its frequency support is near ${\bf e_2}$, the basis vector $(0,1)\in\R^2$. 
	
	Let  $1\leq R\leq \lambda$. We decompose the $\eta$-support of $a^\lambda$ angularly by cutting
	$S^1$ into
	$\approx R^{\frac{1}{2}}$ many arcs $\theta$,
	each  $\theta$  spreading an angle  $\approx R^{-1/2}$. Let $\{\chi_{\theta}\}_\theta\subset C^\infty(\R^2\setminus 0)$ be a family of smooth  cutoff functions associated with the decomposition in the angular direction,
	each of which is  homogeneous of degree $0$,
	such that $\{\chi_\theta\}_{\theta}$ forms a partition  of unity on $\R^2\backslash 0$. So we have
	\begin{equation}
	\sum_{\theta} \chi_{\theta}(\eta)\equiv 1,\;\;\forall \eta \in \mathbb{R}^2\setminus 0.
	\end{equation}
	
	Define
	\begin{align} \label{eq:200a}
 \mathcal{F}^\lambda_\theta f(x,t)
	:=\int_{\R^2}
	e^{i\phi^\lambda(x,t,\eta)}
	a^{\lambda}_\theta(x,t,\eta)\hat{f}(\eta)\,d\eta,
	\end{align}
	where the amplitude reads
	\beq a^{\lambda}_\theta (x,t,\eta):= a(x/\lambda,t/\lambda,\eta)\chi_\theta(\eta).
	\eeq
	Thus,
 $$\mathcal{F}^\lambda f=\sum_\theta \mathcal{F}^\lambda_\theta f.$$
	Now we are ready to state our square function estimates for FIOs.
	\begin{theorem}[Square function estimates for FIOs] \label{sq} Suppose that $\mathcal F^\lambda$ is an operator satisfying the assumptions above. Take $R=\lambda$. Then for any $\eps>0$, $N\in\N$, there exists constants $C_{\eps},C_N>0$ such that
		\begin{equation}\label{eq-sq}
		\|\mathcal{F}^\lambda f\|_{L^4(\R^3)}\leq C_{\eps} \lambda^\varepsilon  \Big\|\Big(\sum_{\theta}|\mathcal{F}^\lambda_\theta f|^2\Big)^{1/2}\Big\|_{L^4(\R^3)}+C_N\lambda^{-N}\|f\|_{L^4(\R^2)}.
		\end{equation} 
		
	\end{theorem}

\begin{remark} \label{general cone remark} Theorem \ref{sq} should be viewed as a variable coefficient generalization of \cite[Theorem 1.1]{GWZ}, which covers the phase function $\phi(x,t,\xi)=x\cdot\xi+t|\xi|$. { Nevertheless, it is not hard to see that a slight modification of their argument implies the same estimate for phase functions of the form $\phi(x,t,\xi)=x\cdot \xi+tp(\xi)$, for any function $p\in C^\infty(\R^2\setminus \{0\})$ homogeneous of degree 1 satisfying ${\rm rank}\,\partial_{\xi\xi}^2 p(\xi)=1$ on the support of the amplitude. In fact, one can see that the only place where the properties of a circular cone are used is  \cite[Lemma 4.2]{GWZ}. We include a proof of this lemma for the general cone case in the Appendix.} In our case, the function $p$ will be allowed to smoothly depend on $x$ and $t$, and thus the cone in the phase space could also vary smoothly from point to point. 
\end{remark}
	It is standard to deduce Theorem \ref{main 1} from Theorem \ref{sq} and a sharp Nikodym maximal function estimate, as demonstrated in \cite{MSSJ}; for readers’ convenience, we include a proof in the Appendix.
	It would be interesting to see if one can prove a higher dimensional analog of Theorem \ref{sq}. Nevertheless, for $d\ge3$, it is  unclear whether we should expect the same conjectured threshold, $\frac{2d}{d-1},$ as in the Euclidean version of this conjecture. Indeed, even though the critical exponents in Conjecture \ref{Conj. 1}, \ref{Conj. 2}, \ref{Conj. 3} are all different for $d\ge3$, there seems to be no argument that directly translates the Kakeya compression phenomena for FIOs to counterexample of square function estimates in this setting. It would be interesting to see if such examples can be constructed.
	
	\subsection{Strategy of Proof}
	It is shown in \cite{BHS} (see also \cite{ILX}) that a variable coefficient decoupling inequality can be deduced from its translation-invariant counterpart by reducing it to a small-scale version through a ``parabolic rescaling and induction on scales" argument. In \cite{GWZ}, the proof of the sharp square function estimate for the cone in $\R^{2+1}$ is also built upon an induction on scales argument, which utilizes a parabolic rescaling type argument. Thus, it might be tempting to think that one should be able to prove Theorem \ref{sq} by reducing matters to small scales and then use the square function estimate of Guth--Wang--Zhang as a black box. Unfortunately, there are significant differences between these two induction on scales arguments which seem to forbid one from doing so. 
	
	In \cite{BHS}, the authors considered ``decoupling constants'' $D(R,\lambda)$ with $\lambda$ being the scale associated with $\mathcal F^\lambda$. Their main result, a variable coefficient decoupling inequality, corresponds to the estimate $D(\lambda^{1-\eps},\lambda)\le C_\eps \lambda^\eps$. The authors showed that at a tiny scale $K\ll\lambda^{\frac12-\eps}$, one could bound $D(K,\lambda)$ favorably by invoking Bourgain--Demeter and a stability lemma. In addition, the authors used a parabolic rescaling argument to bound $D(R,\lambda)$ by $D(K,\lambda)\cdot D(R/K,\lambda/K)$, which allowed them to induct on $R$. This idea was later used in \cite{ILX}, while the argument is completely independent.
	
	In earlier works (see, e.g., \cite{LV,Lee16}), induction on scales arguments were used to prove square function estimates. However, it often seems inefficient to do an induction with respect to a single physical scale. Indeed, Guth, Wang and Zhang formulated stronger square function estimates \cite[Theorem 1.3]{GWZ}, which work better with induction on scales. In their new scheme, two different scales $r\le R$ are involved in the induction, and the ``square function constant'' $S(r,R)$ depends on both $r$ and $R$. The desired square function estimate corresponds to the inequality $S(1,R)\le C_\eps R^\eps$.  The authors then induct on the quotient $\frac Rr$ rather than on $R$ itself.  This formulation is one of the key novel contributions of their work. To be more specific, there are three crucial inputs fed into the induction. The first one is a square function estimate for the parabola, which is used to control $S(r,K)$ for some large constant  $K=K(\eps)$ (for a $K^{-1}$-chunk of the cone). The second one is a Kakeya-type estimate, which controls $S(r,R)$ for $r\ge\sqrt R$. The last ingredient,  Lorentz rescaling, can be used to bound $S(r,R)$ essentially by $S(r,r')\cdot \sup_{s\in[0,1]}S(sr',sR)$ with an intermediate scale $r'$. This allows one to induct on the ratio $R/r$. There are two cases for each step of the induction.  If $r\le K^\frac12$, one goes through $r'=K$ and uses the square function estimate for parabola; if $K^\frac12\le r\le R^\frac12$, one goes through $r'=r^2$ and applies the Kakeya-type estimate to produce a gain. 
	
	Our strategy of proving Theorem \ref{sq} is inspired by all previous work \cite{BHS}, \cite{ILX}, \cite{GWZ}. We shall consider square function constants of the form $S(r,R,\lambda)$, with $\lambda$ being the ambient scale of the FIO. There are also three inputs in our proof. The first one is a small-scale estimate. In fact, if $r\le R\le\lambda^{\frac12-\eps}$, the problem looks almost identical to the translation-invariant case at such a small scale, and therefore we can invoke Guth--Wang--Zhang. The second one is a variable coefficient version of the Kakeya-type estimate, which allows us to control $S(r,R,\lambda)$ for $r\ge R^{\frac{1}{2}+\eps}$. In our case, many key tools available in Euclidean harmonic analysis are absent. We hence develop a new strategy in which we play with multiple different coordinate systems and take advantage of all of them. The authors believe that this strategy is novel and interesting in its own right. With these two estimates in hand, we can bound $S(r,R,\lambda)$ favorably if both $r, R$ are larger or smaller compared to $\lambda^{\frac12-\eps}$. Now we are left with the case when $r\ll \lambda^{\frac12-\eps}\ll R$. We then use a variable coefficient version of the Lorentz rescaling lemma to go through the intermediate scale $\lambda^{\frac12-\eps}$ and close the induction. 
	
	Our paper is organized as follows. In Section \ref{S setup}, we introduce basic reductions and state our main estimate Theorem \ref{iteration}. In Section \ref{S proof}, we present an induction on scale argument and reduce the proof of Theorem \ref{iteration} to the proof of three lemmas. In Section \ref{S geometry}, we collect some geometric facts tied to an important  change of variable map that will become handy in our proof of the lemmas. In the last three sections, we prove our lemmas.
	
	{\bf Acknowledgement.} This project was supported by the National key R\&D program of China No. 2022YFA1005700 and 2022YFA1007200. C. Gao was supported by Chinese Postdoc Foundation Grant No. 8206300279. B. Liu was supported by SUSTech start-up counterpart Y01286235. C. Miao was partially supported by NSF China grant No.11831004.  Y. Xi was partially supported by NSF China grant
	No. 12171424 and the Fundamental Research Funds for the Central Universities 2021QNA3001. The authors would like to thank Ruixiang Zhang for suggesting this problem. The authors would like to thank Larry Guth and Hong Wang for some helpful comments. The authors would like to thank the anonymous referees for their valuable comments and suggestions, which have greatly improved the presentation of this paper. 
	\subsection{Notation}
	
	{ For any arc $\theta\subset S^1$, we use ${\rm d}(\theta)$ to denote its aperture and $\eta_\theta\in S^1$ to denote its center.   For each aperture $s>0$, we fix a collection $\mathbf S_s$ of $10s^{-1}$ arcs to form a covering of $S^1$ with bounded overlaps. }
	
	For $R\geq 1$, $B_R(z)$ denotes a ball centered at $z$ with radius $R$. We abbreviate $B_R=B_R(z)$ when the center is not of particular interest. { We use $w_{B_R}$ to denote a Schwartz weight function that is adapted to $B_R=B_R(z_0)$ so that $w_{B_R}\ge1$ on $B_R$ and decays rapidly away from it.}

	For non-negative quantities $A$ and $B$, we shall write $A\lesssim B$, if there is an absolute constant $C>0$ such that $A\le CB$. 
	We shall write $A\approx B$, if $\frac1C B\le A\le CB$ for some absolute constant $C>1$. Suppose $A(R)$ and $B(R)$  are now functions of taking non-negative values. We write $A(R)\lessapprox_\eps B(R)$, if there exists a constant $C_\eps$ depending on $\eps$, such that for all $R\ge 1$, we have $A(R)\le C_\eps R^\eps B(R).$ 
	
	We shall use $\ra(R)$ to denote a term that is rapidly decaying in $R>1$, that is, for any $N\in\N$, there exists a constant $C_N$ such that $|\ra (R)|\le C_N R^{-N}$.

	For a vector-valued function $F=(f_i)$, we denote its Jacobian matrix by $\partial_x F:=\frac{\partial F}{\partial x}=\left(\frac{\partial f_i}{\partial x_j}\right)_{ij}$. In particular, if $F$ is real-valued, $\nabla_x F=\partial_{x} F$ is a row vector.

	{  We shall say that a function is essentially supported in a set 
$S$ if the sup norm and $L^1$ norm of this function outside of $S$ are $\ra(R)$ with respect to the underlying parameter $R$.}

	We use $\langle x,y \rangle$ to denote the inner product of vectors. When there is no confusion, we also write $x\cdot y:=\langle x, y\rangle$.
	
	\section{Setup and Main Estimates}
	\label{S setup}
	Given $\lambda\gg 1$, we consider  $\mathcal F^\lambda f=\sum_\theta \mathcal F^\lambda_\theta f$  as defined in \eqref{eq:200a}. For simplicity, we write $$F:=\mathcal F^\lambda f,\, F_\theta:=\mathcal F^\lambda_\theta f.$$  Our goal is to establish the estimate \eqref{eq-sq}, which says
	\begin{equation}\label{eq-sq'}
	\|F\|_{L^4(\R^3)}\leq C_\varepsilon \lambda^\varepsilon  \Big\|\Big(\sum_{\theta}|F_\theta |^2\Big)^{1/2}\Big\|_{L^4(\R^3)}+C_N\lambda^{-N}\|F\|_{L^4(\R^2)}.
	\end{equation}  The reader should keep in mind that our $F$ and $F_\theta$ do not enjoy the Fourier support assumptions in \cite{GWZ}. In fact, they can be categorized by their {\it microlocal support} in the following sense.
{ 
\begin{definition}

  We say a smooth function $F: B_\lambda\rightarrow \R$ has microlocal support essentially contained in the $\lambda^{-1}$-neighborhood of a set $S\subset B_\lambda\times\R^3$,  if for any $0<\varepsilon<1$, $1\le R\le \lambda^{1-\varepsilon}$, and any test function $\psi^\lambda(z,\xi):=\psi(z/\lambda,\xi)$ which equals 1 on the $R^{-1}$-neighborhood of $S$, we have
	$$\frac{1}{(2\pi)^3}\iint_{\mathbb R^3\times\mathbb R^3} e^{ i(w-z)\cdot\xi}\, (1-\psi^\lambda(z, \xi))\, F(z)\,dz\,d\xi = \ra(R)||F||_1.$$
	In addition, if $F$ has microlocal support essentially contained in the $\lambda^{-1}$-neighborhood of $S$, we shall say that the microlocal support of $F$ at $z$ is essentially contained in the $\lambda^{-1}$-neighborhood of the set
	$$S_z:=\{\xi:(z,\xi)\in S\}.$$ 
\end{definition}
}
 Now, for $F=\mathcal F^\lambda f$, the relevant integral is the following.
 \[\frac{1}{(2\pi)^3}\iint_{\mathbb R^3\times\mathbb R^3\times\mathbb R^2} e^{ i[(w-z)\cdot\xi+\phi^\lambda(z,\eta)]}\, (1-\psi^\lambda(z, \xi))\,a^\lambda (z,\eta)\hat{f}(\eta)\,dz\,d\xi
	\, d\eta. \]
 A simple integration by parts argument in $z$ variables yields that $F=\mathcal F^\lambda f$ has microlocal support essentially contained in the $\lambda^{-1}$ neighborhood of the set
 	\begin{equation}\label{C}
 	\{(z,\xi): \xi=\nabla_z\phi^\lambda(z,\eta),(z,\eta)\in {\rm supp}\, a^\lambda({}\cdot{},{}\cdot{})\}.
 \end{equation}
Then the microlocal support of $F$ at $z$ is essentially contained in a $\lambda^{-1}$ neighborhood of the cone
	\begin{equation}\label{C_z}
	C_{z}:=\{\xi: \xi=\nabla_z\phi^\lambda(z,\eta),\eta\in {\rm supp}\, a^\lambda(z,{}\cdot{})\}.
	\end{equation}

	Similarly, the microlocal support of $F_\theta$ at $z$ is essentially contained in the $\lambda^{-1}$-neighborhood of the set
		\begin{equation}\label{C^theta}
		C^\theta_{z}:=\{\xi: \xi=\nabla_z\phi^\lambda(z,\eta),\eta\in {\rm supp}\, a^\lambda_\theta(z,{}\cdot{})\}.
	\end{equation}
 	For fixed $z$, the $R^{-1}$ neighborhood of $C^\theta_{z}$ is approximately an	$1\times R^{-1}\times R^{-1/2}$ plank tangent to the cone $C_z$. The long direction of this plank is given by $\nabla_z\phi^\lambda$. 
	
	To prove \eqref{eq-sq'}, we shall prove a stronger square function estimate which
requires further decomposing the square function with respect to various intermediate scales. In this section, we define a key change of variable process associated with each $\theta$, and then use it to describe the precise square function estimate we aim to prove. This change of variable process should be viewed as a variable coefficient version of the Lorentz transformation. That is, given $\eta_\theta\in S^1$, after this process, the light cone in frequency space would be tangent to the horizontal coordinate plane along the direction $\eta_\theta$. See Figure 2 in Section 7 for a special case.
	\subsection{Change of Variables}\label{changev}
	Let $\theta\subset S^1$ be a cap with center $\eta_\theta\in S^1$. We would like to  change the variables from $z=(x,t)$ to new variables $(u,t)$, such that
	$$\nabla_\eta\phi(z,\eta_\theta)=u.$$
    {  More precisely, suppose that $\nabla_\eta\phi(x_0, t_0, \eta_0)=u_0$ holds. By the condition \eqref{CC1} $$\det\left(\frac{\partial(\nabla_\eta\phi, t,\eta)}{\partial(x,t,\eta)}\right)(x_0,t_0,\eta_0)\neq 0.$$ Thus, by the implicit function theorem, there exists a local diffeomorphism $(u,t)\mapsto(x_\eta(u, t), t)$ between a neighborhood of $(u_0, t_0)$ and a neighborhood of $(x_0,t_0)$ such that
	\begin{equation}\label{eq:change}\nabla_\eta\phi((x_\eta(u,t),t), \eta)=u,\end{equation}
    for each $\eta$ in a neighborhood of $\eta_0$. Since our result is local, the original domain of $(x,t,\eta)$ can be decomposed into finitely many such pieces, and it suffices to work with one of them. For later use, we can also assume that this neighborhood of $\eta_0$ is small enough so that $|\eta-\eta_0|^2\leq c|\eta-\eta_0|$, where $c$ is a small but fixed constant.

    We claim that one can further assume $(x_0, t_0)=(u_0, t_0)=0$. To see this, write $\mathcal{F}f(x+x_0, t+t_0)$ as
    $$\begin{aligned}& \frac{1}{(2\pi)^d}\int_{\R^d} e^{i\phi(x+x_0, t+t_0,\eta)}\,a(x+x_0, t+t_0, \eta)\, \hat{f}(\eta)\,d\eta\\=&\frac{1}{(2\pi)^d}\int_{\R^d} e^{i(\phi(x+x_0, t+t_0, \eta)-\nabla_\eta\phi(x_0,t_0,\eta_0)\cdot\eta)}\,a(x+x_0, t+t_0, \eta)\, \hat{\tilde{f}}(\eta)\,d\eta,\end{aligned}$$
    where $\tilde{f}$ has Fourier transform $$\hat{\tilde{f}}(\eta)=e^{i\nabla_\eta\phi(x_0,t_0,\eta_0)\cdot\eta}\hat{f}(\eta).$$
    Since translation does not change the $L^p$-norm of $\mathcal{F}f$, and $\|\tilde{f}\|_{L^4}\approx \|f\|_{L^4}$ by the multiplier theorem, we conclude that for our use it is equivalent to replace the given phase function $\phi$ by
    $$\varphi(x,t,\eta):=\phi(x+x_0, t+t_0, \eta)-\nabla_\eta\phi(x_0,t_0,\eta_0)\cdot\eta,$$
    which satisfies all required conditions to be a phase function and $\nabla_\eta\varphi(0,\eta_0)=0$, as desired. 

    As a summary, for every $\eta$ in a neighborhood of $\eta_0$, $x_\eta(u,t)$ is smooth in $\eta, u, t$, defined in the unit ball in $\R^d$, of range contained in the unit ball in $\R^d$, and \eqref{eq:change} holds.

	In our proof below, we only consider $\eta=\eta_\theta$, the center of dyadic caps $\theta$, so for convenience, we denote $x_\theta:=x_{\eta_\theta}$. Then we define $x^\lambda_\theta(u,t):=\lambda\, x_\theta(u/\lambda,t/\lambda)$, which has domain $B_\lambda$ in $\R^d$ and range contained in $B_\lambda$ in $\R^d$. By rescaling \eqref{eq:change} we have
	\begin{equation}
	\label{change-to-u}
	\nabla_\eta\phi^\lambda((x_\theta^\lambda(u,t),t), \eta_\theta)=u.
	\end{equation}
     }

	By our assumption, $\phi$ is homogeneous in $\eta$ of degree 1. It follows that
	$$ \phi^\lambda=\langle \nabla_\eta\phi^\lambda, \eta\rangle,$$
	which implies, together with \eqref{eq:change}, 
	$$\phi^\lambda(z,\eta_\theta)=\langle u, \eta_\theta\rangle.$$
	This means, if we change variables from $z$ to $(u,t)$ and denote \begin{equation}\label{def-tilde-phi}\tilde{\phi}_\theta^\lambda(u,t,\eta):=\phi^\lambda(x_\theta^\lambda(u,t), t,\eta),\end{equation} we have, for every $\eta\parallel\eta_\theta$,
	\begin{equation}\label{nable-tilde-phi-1}\nabla_{u, t} \tilde{\phi}_\theta^\lambda(u,t,\eta)=(\eta, 0).\end{equation}

	On the other hand, direct computation on \eqref{def-tilde-phi} gives
	\begin{equation}\label{nable-tilde-phi-2}\nabla_{u, t}\tilde{\phi}_\theta^\lambda=\nabla_z\phi^\lambda\cdot\frac{\partial(x_\theta^\lambda,t)}{\partial(u,t)}.\end{equation} 
	
	Recall that at each $z\in B_\lambda$, the microlocal support of $F_\theta$ at $z$ is essentially contained in the $\lambda^{-1}$-neighborhood of the plank $C^\theta_z$ as in \eqref{C^theta} with its long direction given by $\nabla_z\phi^\lambda$. 
	By comparing \eqref{nable-tilde-phi-1} and \eqref{nable-tilde-phi-2}, it follows that under this change of variables, the resulting rectangular box is contained in the $R^{-1}$-neighborhood of the horizontal coordinate plane with normal vector $(0,0,1)$. Moreover, the flat direction of the cone is mapped to $(\eta, 0)$. {    To be more precise, a simple integration by parts argument yields the following lemma.
	\begin{lemma}\label{supp-tilde-F}
If $0<\varepsilon\ll1$ and $1\leq R\leq \lambda^{1-\varepsilon},$ let $$\tilde{F}_\theta (u,t):=F_\theta (x_\theta^\lambda(u,t),t).$$
  then its Fourier transform, $\widehat{\tilde{F}_\theta}(\xi)$, is essentially supported in the $R^{-1}$ neighborhood of the set 
  $$\mathcal N_\theta:=\{(\xi',\xi_3): |\xi'|\approx 1,\  |\xi_{3}|\leq CR^{-1},\  dist(\xi'/|\xi'|,\eta_\theta)\leq CR^{-1/2}\}.$$ 
	\end{lemma}
	\begin{proof}
Recall $F_\theta=\mathcal{F}^\lambda f_\theta$, then the Fourier transform of $\tilde{F}_\theta$ reads
\begin{equation*}
\begin{aligned}
\iint  e^{-i(u,t)\cdot \xi+i \phi^\lambda(x_\theta^\lambda(u,t),t,\eta)}a^\lambda_\theta(x_\theta^\lambda(u,t),t,\eta)\,du \,dt\,\hat{f}(\eta)\,d\eta.
\end{aligned}
\end{equation*}
Define $\mathcal{L}(u,t,\xi,\eta):=-\lambda(u,t)\cdot \xi+ \lambda \phi(x_\theta(u,t),t,\eta).$
Define the operator $\mathcal{D}$ as follows:
\beqq
\mathcal{D}:=-i\frac{\nabla_{u,t}\mathcal{L}(u,t,\xi,\eta)}{|\nabla_{u,t} \mathcal{L}(u,t,\xi,\eta)|^2}\cdot \nabla_{u,t}.
\eeqq
Therefore, for each integer $N$, after rescaling in $u,t$, the above integral can be rewritten as 
\begin{equation*}
\begin{aligned}
\lambda^3\iint (\mathcal{D})^N e^{-i\lambda (u,t)\cdot \xi+i \lambda \phi(x_\theta(u,t),t,\eta)}a_\theta(x_\theta(u,t),t,\eta)\,du \,dt\,\hat{f}(\eta)\,d\eta.
\end{aligned}
\end{equation*}
Note that 
$\nabla_{u,t}\mathcal{L}(u,t,\xi,\eta)=-\lambda\xi+\lambda\partial_{u,t}\phi(x_\theta(u,t),t,\eta)$.
Thus  by the above discussion, we know that the set 
$$\big\{\partial_{u,t}\phi(x_\theta(u,t),t,\eta): \eta \in \supp a_\theta(x_\theta(u,t),t,\cdot)\big\}$$
is contained in $\mathcal{N}_\theta $. By choosing $C$ sufficiently large and noting that  $R\leq \lambda^{1-\varepsilon}$, we see that integration by parts will yield a rapidly decaying term ${\rm RapDec(\lambda)}$.

 \end{proof}}
	We remark that this change of variables $(x,t)=(x_\theta(u,t),t)$ is invertible, namely for each cap $\theta$ we have $(u,t)=(u_\theta(x,t),t)$ with
	$$ \frac{\partial (u,t)}{\partial(x, t)}=\left(\frac{\partial(x, t)}{\partial (u,t)}\right)^{-1}.$$
	Then $\frac{\partial (u,t)}{\partial(x, t)}$ sends the rectangle of direction $(\eta,0)$ on the horizontal coordinate plane to a rectangle tangent to the cone $C_z$ whose direction coincides with the flat direction of the cone.
	
	\subsection{Quantitative conditions } In the proof of our main estimate, we run an induction on scales argument with respect to a class of phase functions satisfying the cinematic curvature condition  \eqref{CC1} and \eqref{CC2}. To facilitate the induction process, we need some quantitative upper bounds on the derivatives of the phase and amplitude functions. Let $\varepsilon_0>0$ be a small number. We first assume that the support of the amplitude function $a(x,t,\eta)$ is  contained in $B_{\varepsilon_0}(0)\times       B_{\varepsilon_0}({\bf e_2})$, i.e,
	\beqq
	\begin{aligned}
		{\bf {\bf C_1}:}\ {\rm supp}_{x,t,\eta}\, a(x,t,\eta)\subset B_{\varepsilon_0}(0)\times       B_{\varepsilon_0}({\bf e_2}).
	\end{aligned}
	\eeqq
 Let $L$ be a sufficiently large number depending on the choice of $\varepsilon_0$ and $C_{\rm par}>0$ be a fixed constant.
	\beqq
	\begin{aligned}
		{\bf {\bf C_2}:}\ \|&\partial_\eta^\alpha \partial_{x,t}^\beta \phi(x,t,\eta)\|_{L^\infty(B_{\varepsilon_0}(0)\times B_{\eps_0}({\bf e_2}))}\leq {C}_{\rm par}, (\alpha,\beta)\in \mathbb{N}^2\times \mathbb{N}^3, \\&|\alpha|\leq L,|\beta|\leq L. \;\text{The nonvanishing eigenvalues of }\\
&\partial_{\eta\eta}^2\langle \partial_z\phi(z,\eta), G(z,\eta_0) \rangle|_{\eta=\eta_0}\;\text{fall  into } [1/2,2], \;\text{for}\; (z,\eta_0)\in {\rm supp\,}a(\cdot,\cdot).
 \end{aligned}
  	\eeqq

	For the amplitude function, we assume
	\beqq
	\begin{aligned}
		{\bf {\bf C_3}:}\ \|&\partial_\eta^\alpha \partial_{x,t}^\beta a(x,t,\eta)\|_{L^\infty(B_{\varepsilon_0}(0)\times B_{\eps_0}({\bf e_2}))}\leq C_{\rm par}, (\alpha,\beta)\in \mathbb{N}^2\times \mathbb{N}^3, \\&|\alpha|\leq L,|\beta|\leq L.
	\end{aligned}
	\eeqq

	\subsection{Main Estimates}\label{subsec-def-U-V}

	In the following part, let $s$ denote a dyadic number. Consider a cap $\tau\subset S^1$ with center $\eta_\tau$ and 
	${\rm d}(\tau)=s$, $R^{-1/2}\leq s\leq 1$. Let $V_\tau$ be the $Rs^2\times Rs$-rectangle centered at the origin of the plane, whose short edge is parallel to $\eta_\tau$. Define
	\beq \label{eq:w1}U_\tau:= \{z=(x_\tau^\lambda(u,t),t): u\in V_\tau,|t|\leq R\}.\eeq
 {  We would like to say some words about the notation $x_\tau$. By our change of variables in Section \ref{changev}, one can see that $x_\tau=x_{\theta_\tau}$ where $\theta_\tau$ denotes the central $R^{-1/2}$-cap of $\tau$. This is because they have the same center, namely $\eta_\tau=\eta_{\theta_\tau}$.}
	
	Under this notation, one can tile $(-R, R)^2$ by a collection of translated copies of $V_\tau$ with finite overlaps. Let us call this collection $\mathcal V_\tau$. Therefore, $$\bigcup _{V\in\mathcal V_\tau}V\supset (-R, R)^2.$$ 
  Correspondingly, define 
 $$U:=\{(x_\tau^\lambda(u,t),t):u\in V,|t|\leq R,V\in \mathcal{V}_\tau\}.$$
 An equivalent yet more explicit definition of $U$ can be found in Section \ref{S lorentz}.
 Let $\mathcal{U}_\tau$ be the collection of all such $U$.

 { 
 \begin{lemma}
The collection $\mathcal U_\tau$ is finitely overlapping.
 \end{lemma}

 \begin{proof}Since the rectangular boxes $\{V\times (-R,R)\}_{V\in\mathcal V_\tau}$ are finitely overlapping and each $U$ is obtained from some $V\times (-R,R)$ through a diffeomorphism $(u,t)\mapsto (x_\tau^\lambda(u,t),t)$,  provided that the size of support $a$ is sufficiently small. Thus $\mathcal U_\tau$ is also a covering of $(-R, R)^{2+1}$ with bounded overlaps.\end{proof}}
 
{ For a given $U\in\mathcal U_\tau,$ define
	\begin{equation}\label{def S_U}S_U F(z) :=\begin{cases}
\left(\sum_{\theta\in\mathbf S_{R^{-1/2}}:\,\theta\subset\tau}|F_\theta|^2(z)\right)^{\frac{1}{2}}\Bigg|_U, & \text{if }z\in U,\\
0, &\text{otherwise}.\end{cases}
 \end{equation}} We remark that, as pointed out in \cite{GWZ}, the size of $U$ can be read off of the index $s$ of the sum and the diameter of the physical ball $R$. 	We shall prove the following analog of Theorem 1.5 in Guth--Wang--Zhang at the physical scale $R=\lambda.$
	\begin{theorem}\label{main}
		$$\|F\|_{L^4(B_\lambda)}^4\lessapprox_\eps \sum_{\lambda^{-1/2}\leq s\leq 1}\sum_{\tau\in\mathbf S_s}\sum_{U\in\mathcal U_\tau} |U|^{-1} \|S_U F\|_{L^2(w_{B_\lambda})}^4+\ra(\lambda)\|f\|^4_{L^2(B_\lambda)}.$$
  { Here $s$ denotes a dyadic number} ranging from $\lambda^{-1/2}$ to $1$.
	\end{theorem}
	Let us first see how Theorem \ref{main} implies Theorem \ref{sq}. The proof is almost identical to that in \cite{GWZ}.  
	
	\begin{proof}[Theorem \ref{main} implies Theorem \ref{sq}]
		Given $\tau$ and $U\in\mathcal U_\tau$, by \eqref{def S_U},  we have
		\[\|S_UF\|^2_{L^2(w_{B_\lambda})}\approx\int_U\sum_{\theta\in\mathbf S_{R^{-1/2}}:\,\theta\subset\tau}|F_\theta|^2.\]
		If we invoke the Cauchy--Schwarz inequality, it follows that
		\[|U|^{-1}\|S_UF\|_{L^2(w_{B_\lambda})}^4\le\int_U\Big(\sum_{\theta\in\mathbf S_{R^{-1/2}}:\,\theta\subset\tau}|F_\theta|^2\Big)^2.\]
		Thus Theorem \ref{main} implies that,  
  modulo a $\text{RapDec}(\lambda)\|f\|^4_{L^2(B_\lambda)}$ error, we have
		\begin{align*}
		\|F\|_{L^4(B_\lambda)}^4&\le C_{\eps} \lambda^\eps  \sum_{\lambda^{-1/2}\leq s\leq 1} \sum_{\tau\in\mathbf S_s}\sum_{U\in\mathcal U_\tau} |U|^{-1} \|S_U F\|^4_{L^2(w_{B_\lambda})}\\
		&\le C_{\eps} \lambda^\eps \sum_{\lambda^{-1/2}\leq s\leq 1}\sum_{\tau\in\mathbf S_s} \Big\|\big(\sum_{\theta\in\mathbf S_{R^{-1/2}}:\,\theta\subset\tau}|F_\theta|^2\big)^{\frac{1}{2}}\Big\|^4_{L^4(B_\lambda)}\\&\le C_{\eps} \lambda^\eps \Big\|\big(\sum_{\theta\in\mathbf S_{R^{-1/2}}:\,\theta\subset\tau}|F_\theta|^2\big)^{\frac{1}{2}}\Big\|^4_{L^4(B_\lambda)}.\end{align*}
		Here we used the fact that the sum over $s$ is dyadic.
	\end{proof}
	
	Theorem \ref{main} is convenient in proving square function estimates. However, as discussed in the introduction, we shall prove a slightly stronger version that works better with multi-scale analysis and induction on scales.

	Note that when $U=B_r$, it is nothing but $s=1$, $R=r$ in the above definition, namely
	\[S_{B_r} F=\begin{cases}
\left(\sum_{_{\theta\in\mathbf S_{r^{-1/2}}}}|F_\theta|^2(z)\right)^{\frac{1}{2}}\Bigg|_{B_r}, & \text{if }z\in B_r,\\
0, &\text{otherwise}.\end{cases}\]
	
	Let $S(r,R,\lambda)$ denote the smallest constant so that for any $N\in\mathbb N$, $B_R\subset B_\lambda$, there exists a constant $C_N>0$ such that
	\begin{multline}\label{def-S}\sum_{B_r\subset B_R}|B_r|^{-1}\|S_{B_r} F\|_{L^2({B_R})}^4 \leq S(r, R, \lambda)\sum_{R^{-1/2}\leq s\leq 1}\sum_{\tau\in\mathbf S_s}\sum_{U\in\mathcal U_\tau} |U|^{-1} \|S_U F\|_{L^2(w_{B_R})}^4\\+C_N\lambda^{-N}\|f\|_{L^2}^4.
	\end{multline}
	It should be understood that $S(r,R,\lambda)$ is uniform across all $\phi$ and $a$ satisfying the cinematic curvature condition and ${\bf {\bf C_1}}$, ${\bf {\bf C_2}}$, ${\bf {\bf C_3}}$. Then to prove Theorem \ref{main} it suffices to prove the following.
	
	\begin{theorem} \label{iteration}For every $\eps>0$, there exists $C_{\eps}>0$ such that 
		$$S(1, \lambda^{1-\eps}, \lambda)\leq C_{\eps} \lambda^{\eps}.$$
	\end{theorem}
	\begin{proof}[Theorem \ref{iteration} implies Theorem \ref{main}] Take $R=\lambda^{1-\eps/100}$, $r=1$, then by Theorem \ref{iteration} we have $S(1, R, \lambda)\leq C_{\eps} \lambda^{\eps/100}.$ Note that
		
		$$\|F\|_{L^4(B_\lambda)}^4\le\lambda^{\eps/30}\sup_{B_R\subset B_\lambda}\|F\|_{L^4(B_R)}^4. $$
		Let $B'_R\subset B_\lambda$ be a ball such that $\|F\|_{L^4(B'_R)}^4\ge\frac12\sup_{B_R\subset B_\lambda}\|F\|_{L^4(B_R)}^4$, then we have
		
		$$\|F\|_{L^4(B_\lambda)}^4\lesssim\lambda^{3\eps/100}\|F\|_{L^4(B'_R)}^4.$$
		Noticing that,
		\begin{multline}
		\|F\|_{L^4(B'_R)}^4\lesssim\sum_{{B_1}\subset B'_R}\|F\|_{L^4(B_r)}^4\lesssim\sum_{B_1\subset B'_R}|B_1|^{-1}\|S_{B_1} F\|_{L^2(B_1)}^4\\ \le  S(1, R, \lambda)\sum_{\lambda^{-1/2}\leq s\leq 1}\sum_{\tau\in\mathbf S_s}\sum_{U\in\mathcal U_\tau} |U|^{-1} \|S_U F\|_{L^2(w_{B'_R})}^4+C_{N+3}R^3\lambda^{-N-3}\|f\|^4_{L^2(B_\lambda)}\\
		\le C_{\eps} \lambda^{\eps} \sum_{\lambda^{-1/2}\leq s\leq 1}\sum_{\tau\in\mathbf S_s}\sum_{U'\parallel U_\tau} |U'|^{-1} \|S_{U'} F\|_{L^2(w_{B_\lambda})}^4+C'_{N}\lambda^{-N}\|f\|^4_{L^2(B_\lambda)}.
		\end{multline}
		Here we used the following facts: firstly, for $U'\supset U,$ where the dimension of $U$ and $U'$ are $R\times Rs\times Rs^2$ and $\lambda\times \lambda s\times \lambda s^2$ respectively, we have
		\[\|S_U F\|_{L^2(w_{B'_R})}^2\le (\lambda/R)^\frac12 \|S_{U'} F\|_{L^2(w_{B_\lambda})}^2=\lambda^{\eps/200} \|S_{U'} F\|_{L^2(w_{B_\lambda})}^2.\]
		Secondly, $|U'|/|U|\approx (\lambda/R)^3=\lambda^{3\eps/100}.$ Finally, the number of $U$ that is contained in $U'$ is bounded by the same number $\lambda^{3\eps/100}$.
	\end{proof}
	
	\section{Outline of the Proof 
	}
	\label{S proof}
	
	The rest of this paper is devoted to the proof of Theorem \ref{iteration}. We need three lemmas, the proofs of which will be given in Section \ref{S small scales}-\ref{S lorentz}.
	
	\begin{lemma}[Small Scale Estimate] \label{small scale}For any $\mu,\,\delta>0$, $1\le r\le R\le \lambda^{\frac1 2-\mu}$, there exists a constant $C_{\mu,\delta}$ such that
		$$S(r, R, \lambda)\leq  C_{\mu,\delta} R^{\delta}(R/r)^\delta.$$
		
	\end{lemma} 
	Lemma \ref{small scale} should be compared to the stability lemma in \cite{BHS} and the small-scale decoupling inequality in \cite{ILX}. The idea is that when the physical scales $(r,R)$ are small compared to $\lambda^\frac12$, the desired square function estimate follows from its Euclidean counterpart rather directly. Indeed, in \cite{GWZ}, the authors considered each $K^{-1}$-chunk of the cone for $K=K(\eps)$ and the associated square function constant $S_K(r,R)$. They proved that $S_K(r,R)\le C_\delta (R/r)^\delta$. Summing over all these chunks gives the Euclidean version of Lemma \ref{small scale}.
	
	\begin{lemma}[Lorentz rescaling]\label{Lorentz} For any $1\le r_1\le r_2\le r_3 \le \lambda^{1-\eps} $, we have
		
		\[S(r_1,r_3,\lambda)\le \log r_2\cdot S(r_1,r_2,\lambda)\max_{r_2^{-1/2}\le s\le1} S(s^2r_2,s^2r_3,s^2\lambda).\]
		
	\end{lemma}	
	Lemma \ref{Lorentz} is the variable coefficient analog of the Lorentz rescaling lemma in \cite{GWZ}. The proof follows the framework developed in \cite{ILX}, which is inspired by ideas from \cite{BHS}. In \cite{BHS}, \cite{ILX}, the authors used decoupling inequality at small scales and a rescaling lemma for decoupling constants to prove variable coefficient analogs. As discussed in the introduction,  estimate (2.12) has two distinct scales $(r,R)$ involved. Thus, unlike previous work, Lemma \ref{small scale} and \ref{Lorentz} are not sufficient to close the induction. We also need the following variable coefficient version of the Kakeya-type estimate in \cite{GWZ}.
	\begin{lemma}[Kakeya-type inequality]\label{Kakeya} For any $\nu,\, \eps>0$, $ \lambda^{\frac1{100}}\le r^2\le \lambda$, there exists a constant $C_{\eps,\nu}$ such that
		
		\[S(r_1,r^2,\lambda)\le C_{\eps,\nu} r^\eps,\ \text{ for all }r_1\in[r^{1+\nu},r^2].\]
		
	\end{lemma}  
	
	As in \cite{GWZ}, we will induct on the ratio $R/r$. However, in our case, an additional parameter $\lambda$ is present. As we shall see in the proof, the relation between $r,R$ and an intermediate scale $\lambda^{1/2-\mu}$ will be the key to our proof. Indeed, we shall use Lemma \ref{Kakeya} to handle the case when both $r$ and $R$ are large compared to $\lambda^{1/2-\mu}$. On the other hand, the case when $r$ and $R$ are both small compared to $\lambda^{1/2-\mu}$ will be covered by Lemma \ref{small scale}. Lastly, when $r$ is small and $R$ is large, we shall invoke Lemma \ref{Lorentz} to go through an intermediate scale to close the induction. 
	\begin{proof}[Proof of Theorem \ref{iteration} via three lemmas]	It suffices to prove that \begin{equation}\label{goal}S(r,R,\lambda)\le  C_\eps R^{\eps/2}(R/r)^{\eps/2},\end{equation}
		for all $1\le r\le R\le\lambda^{1-\eps}.$ Note that $S(r,R,\lambda)\le 1$ trivially holds for all $(r,R,\lambda)$ with $\frac Rr=1$. This will serve as our base case.
		
		\noindent{\bf Induction Hypothesis} For $1\le r\le R\le\lambda^{1-\eps}$, we assume that $$S(r',R',\lambda')\le C_\eps(R')^{\eps/2} (R'/r')^{\eps/2}$$ for all $(r',R',\lambda')$ satisfying  $\lambda'\le\lambda$ and $\frac {R'}{r'}\le\frac{R}{2r}.$

		We choose $\mu=\eps^2/4$, and consider intermediate scales $R_\mu=\lambda^{\frac12-\mu}$ and $R_{2\mu}=\lambda^{\frac12-2\mu}$.  Lemma \ref{small scale} indicates that \eqref{goal} holds if $R\le R_\mu$.  Now for $R_\mu\le R\le R^2_{2\mu}= \lambda^{1-\eps},$ if we also have that $r\ge R_{2\mu}$, then ${\sqrt R}^{1+\eps}\le r\le R.$ Applying Lemma \ref{Kakeya} with $\nu=\eps$ gives \[S(r,R,\lambda)\le C_\eps R^{\eps/2}, \] and thus \eqref{goal} follows from Lemma \ref{Kakeya} directly in this case.

		Finally, we assume that $r<R_{2\mu}$, then $\frac{R}{R_\mu}\le\frac{R}{2r}$ for sufficiently large $\lambda$. If we invoke Lemma \ref{Lorentz} with $r_1=r$, $r_2=R_\mu$, $r_3=R$, we have 
		
		\[S(r,R,\lambda)\le \log \lambda\cdot S(r,R_\mu,\lambda)\max_{R_\mu^{-1/2}\le s\le1} S(s^2R_\mu,s^2R,s^2\lambda).\]
		Using Lemma \ref{small scale} and our induction hypothesis, we have
		
		\[S(r,R,\lambda)\le \log \lambda\cdot C_{\mu,\delta}R_\mu^\delta({R_\mu}/{r})^\delta  C_\eps R^{\eps/2}(R/R_\mu)^{\eps/2}.\]
		
		If we choose $\delta(\eps)$ to be small enough so that $\log\lambda\cdot C_{\mu,\delta}R_\mu^\delta(R_\mu/r)^{\delta-\eps/2}\le 1 $, then \eqref{goal} holds. The induction is closed.
	\end{proof}

	\section{Geometry in Change of Variables}\label{S geometry}
	In Section \ref{changev}, we introduced the change of variables between $(x,t)$ and $(u,t)$. In fact there is a coordinate system $(x_\theta(u,t),t)$ for each $\eta_\theta\in S^1$. These changes of variables are crucial for our analysis, especially in the proof of Lemma \ref{Kakeya}. 
	
	Each coordinate system has its own advantages. Under the original coordinates $(x,t)$, the frequency lies near a cone, so we can play tricks with the cone geometry. This does not work under the coordinates $(u,t)$, since these coordinates depend on the choice of $\eta_\theta$. On the other hand, under coordinates $(u,t)$ associated with $\theta\subset \tau$, the corresponding set 
	\begin{equation}\label{tilde-U} \tilde{U}_\tau:=\{(u,t): (x^\lambda_\tau(u,t), t)\in U_\tau\}=V_\tau\times (-R, R)\end{equation}
	is a $1\times R^{1/2}\times R$-plank, with the property that $\tilde{U}_\tau-\tilde{U}_\tau$ is essentially itself. This, in general, fails on $U_\tau$, which may be curved.
	
	We would like to take advantage of all the different coordinate systems. In this section, we discuss the connection between their geometries carefully. Recall from \eqref{nable-tilde-phi-2} that at each point $(x,t)$, the change from microlocal frequency
	to the Fourier frequency  is given by the matrix
	$$\frac{\partial (u^\lambda_{\theta}(x,t),t)}{\partial(x,t)}.$$
	See Figure \ref{fig1}.
	
	\begin{figure}
		\centering
		\includegraphics[width=.80\textwidth]{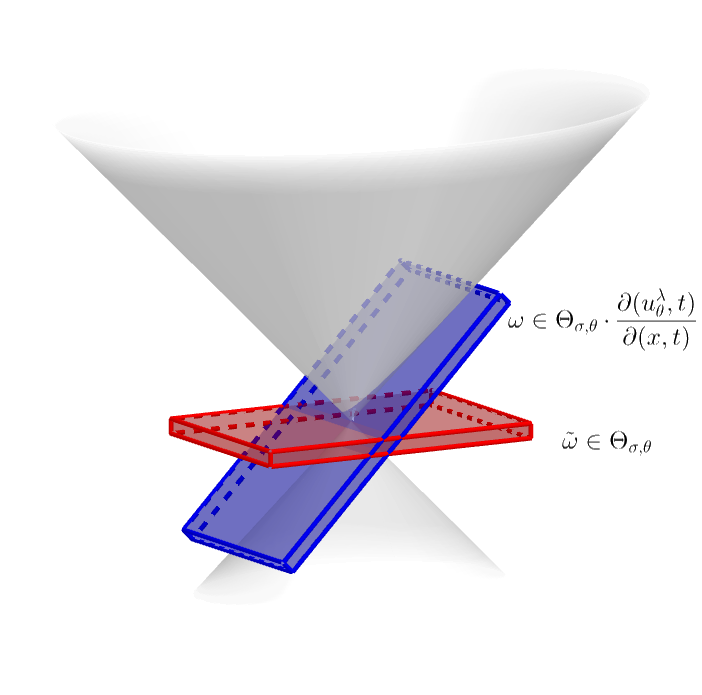}
		\caption{$\frac{\partial (u^\lambda_{\theta}(x,t),t)}{\partial(x,t)}$ acting on $\Theta_{\sigma,\theta}$}
		\label{fig1}
	\end{figure}
	
	For any dyadic number $R^{-1/2}\leq \sigma\leq 1$, denote $\Theta_{\sigma, \theta}$ the rectangular box of dimensions $\sigma^2\times \sigma R^{-1/2}\times R^{-1}$, centered at the original, whose $\sigma^2$-edge has direction $(\theta,0)\in S^2$, $\sigma R^{-1/2}$-edge has direction $(\theta^\perp,0)$, and $R^{-1}$-edge has direction $(0,0,1)$. By Lemma \ref{supp-tilde-F} and basic properties of Fourier transforms of products, these planks ``on the ground" are where the frequency of $|\tilde F_\theta|^2$ lies. Let $\tau\subset S^1$ be a $\sigma^{-1} R^{-1/2}$-cap.  We shall define $\Theta_{\sigma, \tau}$ similarly as a box of the same dimensions with orientation associated with $\tau$. The size of $\tau$ is chosen so that $\Theta_{\sigma, \theta}\subset 4\Theta_{\sigma, \tau}$ for any $\theta\subset\tau$. Therefore it makes no difference to replace $\Theta_{\sigma, \theta}$ by $\Theta_{\sigma, \tau}$. 
 
 Before stating our geometric lemmas on change of variables, we would like to remind the reader of the following notation: $x_\tau=x_{\theta_\tau}$ where $\theta_\tau$ denotes the central $R^{-1/2}$-cap of $\tau$. This is because they have the same center, namely $\eta_\tau=\eta_{\theta_\tau}$. Similarly $u_\tau=u_{\theta_\tau}$.
 
 Our first geometric observation is, when $\sigma$, $\tau$, $(x, t)$, $\tilde{\omega}\in \Theta_{\sigma, \tau}$ are fixed, the corresponding frequency under the coordinates $(x,t)$ is also essentially fixed when $\theta$ runs over $\tau$.
	\begin{lemma}\label{down-up}
		Fix $(x, t)$, $\sigma$, $\tau$ and $\tilde{\omega}\in \Theta_{\sigma, \tau}$. Then for any $\theta_1, \theta_2\subset\tau$,
		$$\left|\tilde{\omega}\cdot\frac{\partial (u^\lambda_{\theta_1}(x,t),t)}{\partial(x,t)}-\tilde{\omega}\cdot\frac{\partial (u^\lambda_{\theta_2}(x,t),t)}{\partial(x,t)}\right|\lesssim R^{-1}.$$
	\end{lemma}
	\begin{proof}
		Recall from Section \ref{changev} that for any $\eta\parallel\eta_\theta$ we have
		$$\nabla_\eta\phi(x,t, \eta)=u_\theta(x,t),\ \phi(x,t,\eta)=\langle u_\theta(x,t), \eta\rangle.$$

		Take $\nabla_\eta$ on both sides of the second equation. Together with the first equation, it follows that
		$$u_\theta=\nabla_\eta\phi=u_\theta+\eta\cdot \partial_\eta u_\theta$$
		which implies for all $(x,t)$, \begin{equation}\label{radial-derivative-u}\eta\cdot \partial_\eta u_\theta=0.\end{equation}
		
		As $\Theta_{\sigma, \tau}$ is $R^{-1}$-thick, without loss of generality we may assume $\tilde{\omega}\in\{\xi: \xi_3=0\}$. Recall that $\eta_\tau\in S^1$ denotes the center of the arc $\tau.$ We write $\tilde{\omega}$ as the sum of two vectors whose directions are $\eta_\tau, \eta_\tau^\perp$, say
		$$\tilde\omega= \tilde\omega_l+\tilde\omega_s,$$
		with $\tilde\omega_l\parallel \eta_\tau$, $|\tilde\omega_l|\leq \sigma^2$, and $\tilde\omega_s\perp \eta_\tau$, $|\tilde\omega_s|\leq \sigma R^{-1/2}$. 
		
		Since $\tilde\omega_l\cdot\partial_\eta u_{\tau}=\eta_\tau\cdot\partial_\eta u_{\tau}=0$ for all $(x,t)$, it follows that 
		$$\tilde\omega_l\cdot\partial_{x_1}\partial_\eta u_{\tau}=\tilde\omega_l\cdot\partial_{x_2}\partial_\eta u_{\tau}=\tilde\omega_l\cdot\partial_t\partial_\eta u_{\tau}=0$$
		and therefore
		$$\left|(\tilde{\omega}_l, 0)\cdot\left(\frac{\partial (u^\lambda_{\theta_1}(x,t),t)}{\partial(x,t)}-\frac{\partial (u^\lambda_{\theta_2}(x,t),t)}{\partial(x,t)}\right)\right|\lesssim [{\rm d}(\tau)]^2\cdot |\tilde{\omega}_l|\lesssim R^{-1}.$$
		
		For $\tilde\omega_s$, one simply has $$\left|(\tilde{\omega}_s, 0)\cdot\left(\frac{\partial (u^\lambda_{\theta_1}(x,t),t)}{\partial(x,t)}-\frac{\partial (u^\lambda_{\theta_2}(x,t),t)}{\partial(x,t)}\right)\right|\lesssim {\rm d}(\tau)\cdot|\tilde{\omega}_s|\lesssim R^{-1}.$$
	\end{proof}
	
	The same holds when the direction of change of variables is reversed.
	\begin{lemma}\label{up-down}
		Fix $(u, t)$, $\sigma$, $\tau$ and $\omega\in \Theta_{\sigma, \tau}\cdot \frac{\partial (u^\lambda_{\theta}(x,t),t)}{\partial(x,t)}$, for some $\theta\subset\tau$. Then for any $\theta_1, \theta_2\subset\tau$,
		$$\left|\omega\cdot\frac{\partial(x^\lambda_{\theta_1}(u,t),t)}{\partial(u,t)}-\omega\cdot\frac{\partial (x^\lambda_{\theta_2}(u,t),t)}{\partial(u,t)}\right|\lesssim R^{-1}.$$
	\end{lemma}
	\begin{proof} 
		The proof still relies on \eqref{radial-derivative-u}, as in the previous lemma. By the previous lemma, we may assume $\eta_\tau$ is  the center of the arc $\tau$. Also, as above, we may decompose a  vector in $\Theta_{\sigma, \tau}$ into three parts, and only the  direction parallel to $(\eta_\tau, 0)$ deserves our attention. Therefore we may assume
		$$\omega=(\tilde{\omega}_l,0)\cdot \frac{\partial (u^\lambda_{\tau}(x,t),t)}{\partial(x,t)},$$
		with $\tilde\omega_l\parallel \eta_\tau$, $|\tilde\omega_l|\leq \sigma^2$.
		
		{ By the definition of our change of variables, 
		\begin{equation}\label{inversion-matrix-id}\frac{\partial(u^\lambda_\tau,t)}{\partial(x,t)}(x^\lambda_\tau(u,t),t)\cdot\frac{\partial(x^\lambda_\tau,t)}{\partial(u,t)}(u,t)=Id.\end{equation}
		Take $\partial_{\eta_k}$, $k=1,2$ on both sides of \eqref{inversion-matrix-id}. We obtain
		\begin{equation}\label{matrix-chain-rule}\frac{\partial(u^\lambda_\tau,t)}{\partial(x,t)}\cdot\left(\partial_{\eta_k}\frac{\partial(x^\lambda_\tau,t)}{\partial(u,t)}\right)+A_k\cdot\frac{\partial(x^\lambda_\tau,t)}{\partial(u,t)}=0,\end{equation}
		where the matrix $A_k$ is given by
		$$A_k= \frac{\partial(\partial_{x_1}u^\lambda_\tau,t)}{\partial(x,t)}\,\partial_{\eta_k}(x^\lambda_\tau)_1 + \frac{\partial(\partial_{x_2}u^\lambda_\tau,t)}{\partial(x,t)}\,\partial_{\eta_k}(x^\lambda_\tau)_2,$$
		with $((x^\lambda_\tau)_1,(x^\lambda_\tau)_2)=x^\lambda_\tau$.
		
		Multiply \eqref{matrix-chain-rule} by $(\tilde{\omega}_l, 0)$ from the left. Then the second term becomes
		\begin{equation}\label{second-term}(\tilde{\omega}_l, 0) \cdot\left(\frac{\partial(\partial_{x_1}u^\lambda_\tau,t)}{\partial(x,t)}\,\partial_{\eta_k}(x^\lambda_\tau)_1 + \frac{\partial(\partial_{x_2}u^\lambda_\tau,t)}{\partial(x,t)}\,\partial_{\eta_k}(x^\lambda_\tau)_2\right)\cdot\frac{\partial(x^\lambda_\tau,t)}{\partial(u,t)}.\end{equation}
		By taking $\partial_{x_i}\partial_{x_j}$,  $\partial_{x_i}\partial_t$ on both sides of  \eqref{radial-derivative-u} and noting the assumption $\tilde\omega_l\parallel \eta_\tau$, one can see that \eqref{second-term} in fact vanishes.
		
		We just proved that, after multiplying \eqref{matrix-chain-rule} by $(\tilde{\omega_l}, 0)$ from the left, the second term vanishes. Then the remaining first term gives
        $$0=(\tilde{\omega}_l,0)\cdot \frac{\partial (u^\lambda_{\tau},t)}{\partial(x,t)}\cdot\left(\partial_{\eta_k}\frac{\partial(x^\lambda_\tau,t)}{\partial(u,t)}\right)=\omega\cdot\left(\partial_{\eta_k}\frac{\partial(x^\lambda_\tau,t)}{\partial(u,t)}\right), \ k=1,2.$$}
		Hence, by the mean value theorem,
		$$\left|\omega\cdot\frac{\partial(x^\lambda_{\theta_1}(u,t),t)}{\partial(u,t)}-\omega\cdot\frac{\partial (x^\lambda_{\theta_2}(u,t),t)}{\partial(u,t)}\right|\lesssim [{\rm d}(\tau)]^2\cdot |\omega| \lesssim R^{-1},$$
		for every $\theta_1, \theta_2\subset\tau$, as desired.
	\end{proof}
	
	We also claim that, when $(x,t)$ is fixed, $(u^\lambda_\theta(x,t), t)$ lies in the same $\tilde{U}\in \tilde{\mathcal{U}}_\tau$ for all $\theta\subset\tau$. As $t$ is fixed, it suffices to consider $u^\lambda_{\theta_1}-u^\lambda_{\theta_2}$. Recall from Section \ref{subsec-def-U-V} that $V_\tau$ in the plane is the $\sigma^{-2}\times \sigma^{-1} R^{1/2}$-rectangle centered at the origin with the short edge parallel to $\eta_\tau$, and  both $U_\tau$ defined in \eqref{eq:w1} and $\tilde{U}_\tau$ defined in \eqref{tilde-U} are associated with the same $V_\tau$. Then our claim is reduced to the following lemma.
	\begin{lemma}\label{wandering-around}
		Fix $(x, t)\in B_R$. Then for any $\theta_1, \theta_2\subset\tau$,
		$$u^\lambda_{\theta_1}(x,t)-u^\lambda_{\theta_2}(x,t)\in 4 V_\tau.$$
	\end{lemma}
	\begin{proof}
		Without loss of generality we  assume that $B_R$ is centered at the origin and  $\phi(0,0,\eta)=0$ for any $\eta$, so $D^\alpha_\eta\phi$ always vanishes when $(x,t)=(0,0)$. Also by \eqref{change-to-u}, with $\theta=\eta/|\eta|$, we have
		\[\nabla_\eta\phi^\lambda(x,t,\eta)=u_\theta^\lambda(x,t).\]
		Therefore all derivatives of $u^\lambda_\theta$ in $\eta$ vanish when $(x,t)=(0,0)$. In particular, this implies
		$$|D_\eta^\alpha u^\lambda_\theta(x,t)|=|D_\eta^\alpha u^\lambda_\theta(x,t)-D_\eta^\alpha u^\lambda_\theta(0,0)|\leq R\, \|D_\eta^\alpha \nabla_{x,t}\,u^\lambda_\theta\|_{L^\infty} \lesssim_\alpha R. $$
		
		Then, along the long edge direction of $V_\tau$,
		$$|\langle u^\lambda_{\theta_1}-u^\lambda_{\theta_2}, \tau^\perp\rangle|\leq|u^\lambda_{\theta_1}-u^\lambda_{\theta_2}|\lesssim \|\nabla_\eta\,u^\lambda_\tau\|_{L^\infty}\cdot {\rm d}(\tau)\lesssim \sigma^{-1} R^{1/2};$$
		for the short edge direction of $V_\tau$, which is $\eta_\tau$, we invoke \eqref{radial-derivative-u} and obtain
		$$|\langle u^\lambda_{\theta_1}-u^\lambda_{\theta_2}, \tau\rangle|\lesssim \|\nabla_\eta\nabla_\eta \, u^{\lambda}_\tau\|_{L^{\infty}}\cdot [{\rm d}(\tau)]^2\lesssim \sigma^{-2},$$
		as desired.
	\end{proof}
	
	\section{Small Scales}
	\label{S small scales}
	In this section, we prove Lemma \ref{small scale} by exploiting a phenomenon known as a ``stability lemma'' in the  literature. The authors of \cite{BHS} used it to deduce small-scale decoupling estimates in the proof of Conjecture \ref{Conj. 3} for odd $d$.  They proved a precise stability lemma using Taylor expansions, which allowed them to approximate $\mathcal F^\lambda$ by an extension operator at tiny scales. Later, inspired by \cite{BHS}, authors in \cite{ILX} developed an independent argument to prove a microlocal  decoupling inequality. The main idea is that at tiny scales $R\ll \lambda^\frac12$, the collection of functions $\{w_{B_R}F_\theta : {\rm d}(\theta)=R^{-1/2}\}$ essentially satisfies the Fourier support assumption in \cite{GWZ}, and thus enjoys the square function estimates proved there. Indeed, it is not hard to see that at such a tiny scale, the varying conic surface in phase space remains within a $R^{-1}$-neighborhood of a fixed cone, and thus the associated tubes $U$ that are curved at a larger scale look straight at scale $R$.   
	In this section, we follow the framework developed in \cite{ILX} to check the Fourier support of restricted $F_\theta$, which seems more illustrative for the readers.

\begin{proof}[Proof of Lemma \ref{small scale}]	Denote the $R^{-1}$-neighborhood of $C_z$ in \eqref{C_z} and $C_z^\theta$ in \eqref{C^theta} by
	\begin{equation}
	C_{z,R  }:=\{\xi: {\rm dist}(\xi,\partial_z\phi^\lambda(z,\eta))\leq R^{-1},\eta\in {\rm supp}\, a^\lambda(z,{}\cdot{})\},
	\end{equation}
	and
	\begin{equation}
	C_{z,R}^\theta:=\{\xi: {\rm dist}(\xi,\partial_z\phi^\lambda(z,\eta))\le R^{-1},\eta\in {\rm supp}\, a_\theta^\lambda(z,{}\cdot{})\}.
	\end{equation}
	
	Fix $\mu>0$. We shall see that when $R$ is smaller than $\lambda^{1/2-\mu}$, Lemma \ref{small scale} can be reduced to the Euclidean case, then the result of Guth--Wang--Zhang \cite{GWZ} applies directly. Fix $\bar{z}\in B_\lambda$ and consider $z\in B_{\lambda^{\mu/2}R}(\bar{z})$, $R<\lambda^{1/2-\mu}$.
	
	{ \begin{lemma}\label{eq:obser}
		There exists an absolute constant $C>0$ such that:
		 for any $z\in B_{\lambda^{\mu/2}R}(\bar{z})$, the set $C_{z,R},C_{z,R}^\theta$ are contained in a $CR^{-1}$-neighborhood of   $C_{\bar{z},R}$ and $C_{\bar{z},R}^\theta$ respectively. 
	\end{lemma}
	
	\begin{proof}
		 Recall $|\det(\partial^2_{zz}\phi^\lambda)|\lesssim 1$, for $C_{z,R}$, it follows from
		$$|\nabla_z\phi^\lambda(z, \eta)-\nabla_z\phi^\lambda(\bar{z}, \eta)|\lesssim|z-\bar{z}|/\lambda\lesssim \lambda^{\mu/2}R/\lambda\le\lambda^{-1/2-\mu+\mu/2}\lesssim R^{-1}.$$  Similarly, we obtain the desired result for $C_{z,R}^\theta$. 
	\end{proof}
	For $1\leq R\leq \lambda^{1/2-\mu}$, let us consider the Fourier support of the restriction of $F,F_\theta$ to the ball $B_R(\bar z)$ using an appropriate weight function. 
 \begin{lemma}\label{eq:le5}
Let $1\leq R\leq \lambda^{1/2-\mu}$. The Fourier support of $w_{B_{\lambda^{\mu/3}R}(\bar z)} F$ and  $w_{B_{\lambda^{\mu/3}R}(\bar z)} F_\theta$ are essentially contained in the $CR^{-1}$-neighborhood of $C_{\bar z,R}$ and $C_{\bar z,R}^\theta$ respectively.
\end{lemma}
 \begin{proof}
Recall $F=\mathcal{F}^\lambda f$. The Fourier transform of $w_{B_{\lambda^{\mu/3}R}(\bar z)} F$ is
\begin{equation*}
\iint e^{-iz\cdot \xi+i\phi^\lambda(z,\eta)}a^\lambda(z,\eta)w_{B_{\lambda^{\mu/3}R}(\bar z)}(z) \hat{f}(\eta)\,dz\,d\eta.
\end{equation*}
If $z\notin B_{\lambda^{\mu/2}R}(\bar z)$, then 
\begin{equation*}
|w_{B_{\lambda^{\mu/3}R}(\bar z)}(z)|={\rm RapDec}(\lambda).
\end{equation*}
Let $\Psi_{B_{\lambda^{\mu/2}R}(\bar z)}$ be a smooth function with compact support in $B_{\lambda^{\mu/2}R}(\bar z)$. We have, modulo a ${\rm RapDec}(\lambda)\|f\|_{L^1}$ error,
\begin{equation}
\begin{aligned}
&\iint e^{-iz\cdot \xi+i\phi^\lambda(z,\eta)}a^\lambda(z,\eta)w_{B_{\lambda^{\mu/3}R}(\bar z)} (z)\,\hat{f}(\eta)\,dz\,d\eta\\&=
\iint e^{-iz\cdot \xi+i\phi^\lambda(z,\eta)}a^\lambda(z,\eta)w_{B_{\lambda^{\mu/3}R}(\bar z)}(z)\Psi_{B_{\lambda^{\mu/2}R}(\bar z)}(z)\, \hat{f}(\eta)\,dz\,d\eta.
\end{aligned}
\end{equation}
By integrating by parts with respect to $z$, one can check that if $\xi\not\in\bigcup_{z\in B_{\lambda^{\mu/2}R}(\bar z)} C_{z,R}$, the integral is rapidly decaying in terms of $\lambda$. We then obtain the desired result by Lemma \ref{eq:obser}.  The corresponding statement regarding $w_{B_{\lambda^{\mu/3}R}(\bar z)}F_\theta$ 
 can be obtained similarly. \end{proof}
 
To sum up, we have reduced estimating $S(r,R,\lambda)$ for $R\le \lambda^{\frac12-\mu}$ to proving corresponding square function estimate associated with the ({fixed}) cone $C_{\bar z}$. To be more precise, if $R\leq \lambda^{1/2-\mu}$,
by Lemma \ref{eq:le5}, it follows that 
\beqq
w_{B_{\lambda^{\mu/3}R}(\bar z)}\mathcal{F}^\lambda f=\chi_{C_{\bar z,R}}(D)(w_{B_{\lambda^{\mu/3}R}(\bar z)}\mathcal{F}^\lambda f)+{\rm RapDec}(\lambda)\|f\|_{L^1},
\eeqq
and 
\beq
\begin{aligned}
w_{B_{\lambda^{\mu/3}R}(\bar z)}\mathcal{F}^\lambda f&=\sum_\theta \chi_{C_{\bar z,R}^\theta}(D)(w_{B_{\lambda^{\mu/3}R}(\bar z)}\mathcal{F}^\lambda f_\theta)+{\rm RapDec}(\lambda)\|f\|_{L^1}\\
&=\sum_\theta \chi_{C_{\bar z,R}^\theta}(D)(w_{B_{\lambda^{\mu/3}R}(\bar z)}\mathcal{F}^\lambda f)+{\rm RapDec}(\lambda)\|f\|_{L^1},
\end{aligned}
\eeq
where $\chi_{C_{\bar z,R}}(D), \chi_{C_{\bar z,R}^\theta}(D)$ are two Fourier multipliers, and $\chi_E$ denotes the characteristic function of $E$. Therefore, modulo a term rapidly decaying in $\lambda$, we have 
\beqq
\chi_{C_{\bar z,R}}(D)(w_{B_{\lambda^{\mu/3}R}(\bar z)}\mathcal{F}^\lambda f)=\sum_\theta \chi_{C_{\bar z,R}^\theta}(D)(w_{B_{\lambda^{\mu/3}R}(\bar z)}\mathcal{F}^\lambda f).
\eeqq
Now, we may apply a version of \cite[Proposition 3.4]{GWZ}\footnote{Indeed, Proposition 3.4 in \cite{GWZ} is an estimate for $S_K(r,R)$, while what we need is the same estimate but for $S(r,R)$. Nonetheless, this estimate is a simple corollary of Proposition 3.4, and its proof is nearly identical to the proof of Theorem 1.3 there.} with respect to the cone $C_{\bar z}$ (see Remark \ref{general cone remark}) to finish the proof of Lemma \ref{small scale}. }\end{proof}

	\section{A Microlocal Kakeya-Type Estimate} \label{S Kakeya}
	In this section, we prove Lemma \ref{Kakeya}. We first reduce it to the following Kakeya-type lemma, which is a microlocal analog of Lemma 1.4 in \cite{GWZ}. The key difference from that in \cite{GWZ} is that we exploit  different coordinate systems and take advantage of all of them. We have discussed the necessity of changing variables in Section \ref{S geometry}.
	
	\begin{lemma}\label{Kakeya-type-estimate}
		Let $G=\sum_{{\rm d}(\theta)=R^{-1/2}}|F_\theta|^2$. Then for any $\eps>0$ there exists $C_\eps>0$ such that
		$$\int_{B_R} |G|^2\leq C_\eps R^\eps \sum_{ R^{-1/2}\leq s\leq 1}\sum_{\tau\in\mathbf S_s}\sum_{U\in\mathcal U_\tau} |U|^{-1}\|S_U F\|_{L^2(w_{B_R)}}^4.$$
	\end{lemma}
	
	\begin{proof}
		[Lemma \ref{Kakeya-type-estimate} implies Lemma \ref{Kakeya}.] {  Let $\psi_{B_R}\in C_0^\infty$ be a non-negative function that equals $1$ on $B_R$ and $0$ outside $2 B_R$. By the definition of $S_{B_{r_1}} F$,
  $$\|S_{B_{r_1}} F\|_{L^2(B_{r_1})}^2\leq \int \psi_{B_{r_1}}\cdot\sum_{{\tau\in\mathbf S_{r_1^{-1/2}}}}|F_\tau|^2.$$
  Since $r_1\leq r^2$, every $r_1^{-1/2}$-cap $\tau$ can be decomposed into a union of $r^{-1}$-caps. This means one can write  
  $$F_\tau=\sum_{\substack{\theta\subset\tau,\\ {\theta\in\mathbf S_{r^{-1}}}}}F_\theta$$
  and
	$$\begin{aligned} \int \psi_{B_{r_1}}\cdot\sum_{{\tau\in\mathbf S_{r_1^{-1/2}}}}|F_\tau|^2=&\int \psi_{B_{r_1}}\cdot\sum_{{\tau\in\mathbf S_{r_1^{-1/2}}}}\left|\sum_{\substack{\theta\subset\tau,\\ \theta\in\mathbf S_{r^{-1}}}}F_\theta\right|^2\\=&\sum_{{\tau\in\mathbf S_{r_1^{-1/2}}}}\sum_{\substack{\theta_1\subset\tau,\\ \theta_1\in\mathbf S_{r^{-1}}}}\sum_{\substack{\theta_2\subset\tau,\\ \theta_2\in\mathbf S_{r^{-1}}}}\int \psi_{B_{r_1}}\cdot F_{\theta_1} \bar{F}_{\theta_2}.\end{aligned}$$
    We shall show that we only need to consider adjacent caps $\theta_1, \theta_2$.
		
		Recall from \eqref{eq:200a} that
		$$F_\theta(z) = \F^\lambda_\theta f(z)=\int_{\R^2}
		e^{i\phi^\lambda(z,\eta)}
		a^{\lambda}_\theta(z,\eta)\hat{f}(\eta)\,d\eta,$$
        so the integral
        $$\begin{aligned}&\int \psi_{B_{r_1}}\cdot F_{\theta_1} \bar{F}_{\theta_2}\\=& \iint\left(\int e^{i( \phi^\lambda(z,\eta_1)-\phi^\lambda(z,\eta_2))}\,a^{\lambda}_{\theta_1}(z,\eta_1)\,\overline{a^{\lambda}_{\theta_2}(z,\eta_2)}\,\psi_{B_{r_1}}(z)\,dz\right)\hat{f}(\eta_1)\overline{\hat{f}(\eta_2)}\,d\eta_1\,d\eta_2\end{aligned}.$$
		To show it has rapid decay when $\theta_1, \theta_2$ are $r^{-1}$-separated, it suffices to show the kernel
  $$\int e^{i (\phi^\lambda(z,\eta_1)-\phi^\lambda(z,\eta_2))}\,a^{\lambda}_{\theta_1}(z,\eta_1)\,\overline{a^{\lambda}_{\theta_2}(z,\eta_2)}\,\psi_{B_{r_1}}(z)\,dz$$
  has rapid decay. This is H\"ormander’s argument. More precisely, the phase function in this kernel satisfies
  $$|\nabla_x(\phi^\lambda(x,t,\eta_1)- \phi^\lambda(x,t,\eta_2))|=|\partial_{x\eta}^2\phi^\lambda\cdot(\eta_1-\eta_2)|+O(|\eta_1-\eta_2|^2).$$
  By condition \eqref{CC1}, the first term is $\approx |\eta_1-\eta_2|$, where the implicit constant only depends on $\phi$, and thus is independent of $\lambda$. We have discussed in Section \ref{changev} that one can assume the $\eta$-support of the amplitude is small enough so that $|\eta_1-\eta_2|^2\ll |\eta_1-\eta_2|$. Therefore the $O(|\eta_1-\eta_2|^2)$ term can be ignored and
  $$|\nabla_x(\phi^\lambda(x,t,\eta_1)- \phi^\lambda(x,t,\eta_2))|\gtrsim |\eta_1-\eta_2|,$$
  which is $\gtrsim r^{-1}$ if $\theta_1, \theta_2$ are $r^{-1}$-separated. Then by integration by parts
  $$\left|\int \psi_{B_{r_1}}\cdot F_{\theta_1} \bar{F}_{\theta_2}\right|= \ra(r_1/r)\|f\|_{L^2}^2=\ra(\lambda)\|f\|_{L^2}^2,$$ 
  where the last equality follows from the assumption $r_1\geq r^{1+\nu}\geq \lambda^{\frac1{200}}$ in Lemma \ref{Kakeya}. 
  
  With this local $L^2$-orthogonality, it follows that
		$$\int\psi_{B_{r_1}}\cdot \left|\sum_{\substack{\theta\subset\tau,\\ \theta\in\mathbf S_{r^{-1}}}}F_\theta\right|^2\lesssim\int \psi_{B_{r_1}}\cdot \sum_{\substack{\theta\subset\tau,\\ \theta\in\mathbf S_{r^{-1}}}}|F_\theta|^2+ \ra(\lambda)\|f\|_{L^2}^2. $$
  
		Then we take sum in $\tau\in\mathbf S_{r_1^{-1/2}}$ to obtain that, up to a  $\ra(\lambda)$ error,
		$$\|S_{B_{r_1}} F\|_{L^2(B_{r_1})}^2\lesssim \sum_{{\tau\in\mathbf S_{r_1^{-1/2}}}}\int \psi_{B_{r_1}}\cdot \sum_{\substack{\theta\subset\tau,\\ \theta\in\mathbf S_{r^{-1}}}}|F_\theta|^2= \int \psi_{B_{r_1}}\cdot \sum_{\theta\in\mathbf S_{r^{-1}}}|F_\theta|^2=\int\psi_{B_{r_1}}\cdot G.$$
		Finally, by Cauchy--Schwarz, modulo a $\ra(\lambda)$ term,
		$$\sum_{B_{r_1}\subset B_r}|B_{r_1}|^{-1}\|S_{B_{r_1}} F\|_{L^2(B_{r_1})}^4\lesssim\sum_{B_{r_1}\subset B_r}\int_{B_{r_1}} |G|^2 =\int_{B_r} |G|^2.$$
		
		Hence Lemma \ref{Kakeya} follows from Lemma \ref{Kakeya-type-estimate}. }
	\end{proof}

	Now it remains to prove Lemma \ref{Kakeya-type-estimate}.

	Recall that in Section 2.3, for each $\theta$ there is an associated change of variables $(x,t)=(x_\theta(u,t),t)$, and we have defined $\tilde{F}_\theta(u,t)=F_\theta(x^\lambda_\theta(u,t), t)$. Denote its inverse as $(u,t):=(u_\theta(x,t),t)$ and $u_\theta^\lambda({}\cdot{}):= \lambda u_\theta({}\cdot/{}\lambda)$.  By Fourier inversion, 
	$$|F_\theta(x,t)|^2 = |\tilde{F}_\theta(u^\lambda_\theta(x,t),t)|^2=\int e^{2\pi i(u^\lambda_\theta(x,t),t)\cdot\tilde{\omega}}\, \left(|\tilde{F}_\theta|^2\right)^\wedge(\tilde{\omega})\,d\tilde{\omega}.$$
	Then
	\begin{equation}\label{Fourier-inversion-1}\int_{B_R}|G|^2=\int_{B_R}\left(\sum_\theta\int e^{2\pi i(u^\lambda_\theta(x,t),t)\cdot\tilde{\omega}}\,  \left(|\tilde{F}_\theta|^2\right)^\wedge(\tilde{\omega})\,d\tilde{\omega}\right)^2dx\,dt.\end{equation}
	
	By Lemma \ref{supp-tilde-F}, the Fourier support of $\tilde{F}_\theta$ lies in $\mathcal{N}_\theta$. By basic properties of Fourier transforms of products, one can see that $\tilde{\omega}$ in \eqref{Fourier-inversion-1} lies in $$\mathcal{N}_\theta+(-\mathcal{N}_\theta),$$ which is contained in a $1\times R^{-1/2}\times R^{-1}$-plank $\Theta_\theta$ centered at the origin, with its $1$-edge parallel to $(\eta_\theta,0)$ and $R^{-1/2}$-edge parallel to $(\eta_\theta^\perp,0)$.
	
	Then we decompose $\Theta_\theta$ as in Section \ref{S geometry}. For any dyadic number $R^{-1/2}\leq \sigma\leq 1$, denote $\Theta_{\sigma, \theta}\subset\Theta_\theta$ as the plank of dimensions $\sigma^2\times \sigma R^{-1/2}\times R^{-1}$, centered at the original of direction, with its $\sigma^2$-edge parallel to $(\eta_\theta, 0)$ and $\sigma R^{-1/2}$-edge parallel to $(\eta_\theta^\perp,0)$. Let $\psi_{\Theta_{\sigma, \theta}}\in C_0^\infty$ which equals $1$ on $\Theta_{\sigma, \theta}$ and $0$ outside $\Theta_{2\sigma, \theta}$. Then define
	$$
	\Omega_{\sigma, \theta}:=\Theta_{\sigma, \theta}\backslash \Theta_{ \frac{\sigma}{2}, \theta},\  \psi_{\Omega_{\sigma, \theta}}:=\psi_{\Theta_{\sigma, \theta}}-\psi_{\Theta_{\sigma/2, \theta}}, \ \text{if } 2R^{-1/2}\leq \sigma<1,
	$$
	and
	$$\Omega_{R^{-1/2}, \theta}:=\Theta_{R^{-1/2}, \theta}\sim B_{R^{-1}},\ \psi_{\Omega_{\sigma, \theta}}:=\psi_{\Theta_{R^{-1/2}, \theta}}.$$

	{ Now we can decompose the domain of $\tilde\omega$ in \eqref{Fourier-inversion-1} to obtain
	\begin{equation}\label{decompose-planks-epsilon-room}\begin{aligned}
	&\int_{B_R}\left(\sum_{\sigma}\sum_{\theta}\int e^{2\pi i(u^\lambda_\theta(x,t),t)\cdot\tilde{\omega}}\, \left(|\tilde{F}_\theta|^2\right)^\wedge(\tilde{\omega})\,\psi_{\Omega_{\sigma, \theta}}(\tilde{\omega})\,d\tilde{\omega}\right)^2dx\,dt\\=&\int\left(\chi_{B_R}\cdot\sum_{\sigma}\sum_{\theta}\int e^{2\pi i(u^\lambda_\theta(x,t),t)\cdot\tilde{\omega}}\, \left(|\tilde{F}_\theta|^2\right)^\wedge(\tilde{\omega})\,\psi_{\Omega_{\sigma, \theta}}(\tilde{\omega})\,d\tilde{\omega}\right)^2dx\,dt.\end{aligned}
	\end{equation}
	
	By Plancherel  in $(x,t)$ it equals
	\begin{equation}\label{ready-for-CS}\begin{aligned}
	    &\int\left(\int_{B_R}e^{-2\pi i(x,t)\cdot\omega}\left(\sum_{\sigma}\sum_{\theta}\int e^{2\pi i(u^\lambda_\theta(x,t),t)\cdot\tilde{\omega}}\, \left(|\tilde{F}_\theta|^2\right)^\wedge(\tilde{\omega})\,\psi_{\Omega_{\sigma,\theta}}(\tilde{\omega})\,d\tilde{\omega}\right)dx\,dt\right)^2d\omega\\=&\int\left(\int_{B_R}\sum_{\sigma}\sum_{\theta}\int e^{2\pi i((u^\lambda_\theta(x,t),t)\cdot\tilde{\omega}-(x,t)\cdot\omega)}\, \left(|\tilde{F}_\theta|^2\right)^\wedge(\tilde{\omega})\,\psi_{\Omega_{\sigma,\theta}}(\tilde{\omega})\,d\tilde{\omega}\,dx\,dt\right)^2d\omega.\end{aligned}\end{equation}}
	
	Notice that, by integration by parts in $(x,t)$, it suffices to consider $(\theta, (x,t), \omega, \tilde{\omega})$ satisfying
	\begin{equation}\label{cut-off-on-4}\left|\tilde{\omega}\cdot\frac{\partial (u^\lambda_\theta(x,t),t)}{\partial(x,t)}-\omega\right|\lessapprox_\eps R^{-1}.\end{equation}
	We choose a smooth cutoff function $\chi_\theta(x,t,\omega,\tilde{\omega})$ which equals $1$ for all $(x,t,\omega,\tilde\omega)$ satisfying \eqref{cut-off-on-4}.
	
	For each fixed $\sigma$, divide the circle into $\sigma^{-1}R^{-1/2}$-arcs $\tau$. When $(x,t), \sigma, \tau, \omega$ are fixed and $\theta\subset\tau$, \eqref{cut-off-on-4} and Lemma \ref{up-down}  imply that $\tilde{\omega}$ essentially lies in a ball of radius $R^{-1}$. { Therefore there is no extra cost to replace $\chi_\theta$ by $\chi_{\theta_\tau}$, where $\theta_\tau$ again denotes the central $R^{-1/2}$-cap of $\tau$. For convenience we denote $\chi_\tau:=\chi_{\theta_\tau}$.} Then the integral \eqref{ready-for-CS} becomes
	\begin{equation}\label{right-before-CS}\int\left(\int_{B_R}\sum_{\sigma}\sum_{\tau\in \mathbf{S}_{\sigma^{-1}R^{-1/2}}}\int_{\Omega_{\sigma,\tau}}\chi_\tau\sum_{\theta\in\mathbf S_{R^{-1/2}}:\,\theta\subset\tau} \cdots d\tilde{\omega}\,dx\,dt\right)^2d\omega.\end{equation}
	
	Now we apply Cauchy--Schwarz to \eqref{right-before-CS}. More precisely, it is bounded above by an integral on $\omega$, with integrand the product of
	\begin{equation}
	\label{integrand-1}
	\int_{B_R}\sum_\sigma\sum_{\tau \in \mathbf{S}_{\sigma^{-1}R^{-1/2}}}\int_{\Omega_{\sigma,\tau}}\chi_{\tau}\,d\tilde{\omega}\,dx\,dt
	\end{equation}
	and
	\begin{equation}
	\label{integrand-2}
	\int_{B_R}\sum_\sigma\sum_{\tau}\int_{\Omega_{\sigma,\tau}}\chi_{\tau}\left|\sum_{\theta\in\mathbf S_{R^{-1/2}}:\,\theta\subset\tau} e^{2\pi i((u^\lambda_\theta(x,t),t)\cdot\tilde{\omega}-(x,t)\cdot\omega)}\,\left(|\tilde{F}_\theta|^2\right)^\wedge(\tilde{\omega})\right|^2d\tilde{\omega}\,dx\,dt.
	\end{equation}
	
	For the first factor \eqref{integrand-1}, we have seen that when $\omega, (x,t), \sigma, \tau$ are fixed, \eqref{cut-off-on-4} and Lemma \ref{up-down}  imply that $\tilde{\omega}$ essentially lies in a ball of radius $R^{-1}$. Therefore \eqref{integrand-1} is
	\begin{equation}\label{int-tilde-omega}
	\lessapprox_\eps R^{-3}\sum_\sigma\int_{B_R}\sum_\tau\chi^*_{\tau} (x,t,\omega)\,dx\,dt,
	\end{equation}
	where $$\chi^*_{\tau} (x,t,\omega)=\sup_{\tilde\omega\in \Omega_{\sigma,\tau}}\chi_{\tau} (x,t, \omega,\tilde\omega).$$
	By \eqref{cut-off-on-4} one can see that $\chi^*_\tau$ vanishes unless $\omega$ essentially lies in the $R^{-1}$-neighborhood of 
	$$\Omega_{\sigma, x,t}:=\bigcup_\theta\,\Omega_{\sigma, \theta}\cdot\frac{\partial (u^\lambda_\theta(x,t),t)}{\partial(x,t)}.$$

	Here comes the key observation from Guth--Wang--Zhang (see Lemma 4.2 in \cite{GWZ}). Let $\mathcal{N}_{R^{-1}}(\Omega_{\sigma, x,t})$ denote the $R^{-1}$ neighborhood of $\Omega_{\sigma, x,t}$. When $\sigma$, $(x,t)$, $\omega\in \mathcal{N}_{R^{-1}}(\Omega_{\sigma, x,t})$ are fixed, there are only finitely many $\tau$ with $\chi^*_\tau$ non-vanishing.   This is proved in \cite{GWZ} for the standard cone $(\xi, |\xi|)\in\R^3$, while as we explained right after Theorem \ref{sq}, it remains valid on a general cone with its curvature bounded from above and below. In particular, it applies to our cones for each fixed $(x,t)$, thanks to our cinematic curvature assumption. For completeness, we provide proof of a generalization of \cite[Lemma 4.2]{GWZ} to a general cone in the Appendix (Lemma \ref{general cone}).
	
	The discussion above implies that \eqref{int-tilde-omega} is bounded above by
	$$R^{-3}\sum_\sigma \int_{B_R}\,dx\,dt\lesssim \log R.$$

	We just showed that the factor \eqref{integrand-1} is $\lessapprox_\eps 1$. It remains to consider the integral on $\omega$ with integrand \eqref{integrand-2}, which equals
	$$\int\int_{B_R}\sum_\sigma\sum_{\tau\in \mathbf{S}_{\sigma^{-1}R^{-1/2}}}\int_{\Omega_{\sigma,\tau}}\chi_{\tau}\left|\sum_{\theta\in\mathbf S_{R^{-1/2}}:\,\theta\subset\tau} e^{2\pi i(u^\lambda_\theta(x,t),t)\cdot\tilde{\omega}}\,\left(|\tilde{F}_\theta|^2\right)^\wedge(\tilde{\omega})\right|^2d\tilde{\omega}\,dx\,dt\,d\omega$$
	
	When $(x,t),\ \sigma,\ \tau,\ \tilde{\omega}$ are fixed, \eqref{cut-off-on-4} and Lemma \ref{down-up} imply that $\chi_\tau$ vanishes unless $\omega$ essentially lies in a ball of radius $R^{-1}$. Therefore by integrating $d\omega$ it is bounded above by
	$$R^{-3}\int_{B_R}\sum_\sigma\sum_{\tau\in \mathbf{S}_{\sigma^{-1}R^{-1/2}}}\int_{\Theta_{\sigma, \tau}}\left|\sum_{\theta\in\mathbf S_{R^{-1/2}}:\,\theta\subset\tau} e^{2\pi i(u^\lambda_\theta(x,t),t)\cdot\tilde{\omega}}\,\left(|\tilde{F}_\theta|^2\right)^\wedge(\tilde{\omega})\right|^2d\tilde{\omega}\,dx\,dt.$$
	
	Fix $(x,t)$ and apply Plancherel. It becomes
	$$ R^{-3}\int_{B_R}\sum_\sigma\sum_{\tau}\int\left|\sum_{\theta\in\mathbf S_{R^{-1/2}}:\,\theta\subset\tau}\int e^{2\pi i((u^\lambda_\theta(x,t),t)+(u_1, t_1))\cdot\tilde{\omega}}\left(|\tilde{F}_\theta|^2\right)^\wedge(\tilde{\omega})\,\psi_{\Theta_{\sigma, \tau}}(\tilde{\omega})\,d\tilde{\omega}\right|^2\,du_1dt_1dxdt$$
	\begin{equation}\label{with-kernel}= R^{-3}\int_{B_R}\sum_\sigma\sum_{\tau}\int\left|\sum_{\theta\in\mathbf S_{R^{-1/2}}:\,\theta\subset\tau}\int K_\theta(x,t,u_1,t_1,u_2,t_2)\,\left|\tilde{F}_\theta(u_2,t_2)\right|^2du_2dt_2\right|^2du_1dt_1dxdt,\end{equation}
	where the kernel $K_\theta$ is given by
	\begin{equation}\label{kernel-K-theta}K_\theta(x,t,u_1,t_1,u_2,t_2)=\int e^{2\pi i((u^\lambda_\theta(x,t),t)+(u_1,t_1)-(u_2, t_2))\cdot\tilde{\omega}}\,\psi_{\Theta_{\sigma,\tau}}(\tilde{\omega})\,d\tilde{\omega}.\end{equation}
	Notice that $\Theta_{\sigma,\tau}$ and $\tilde{U}_\tau$ defined in \eqref{tilde-U} are dual rectangular boxes to each other. Therefore
	\begin{equation}\label{upper-bound-K-theta}
	\|K_\theta\|_{L^\infty}\lesssim |\Theta_{\sigma,\tau}|=|\tilde{U}_\tau|^{-1}\approx |U_\tau|^{-1},
	\end{equation}
	and by integration by parts $K_\theta$ has rapid decay unless essentially
	\begin{equation}\label{essentially-in-U-tau}(u^\lambda_\theta(x,t),t)+(u_1, t_1)-(u_2, t_2)\in\Theta_{\sigma,\tau}^*= \tilde{U}_\tau.\end{equation}
	
	{ We cover the domain of $(u_1,t_1)$ in \eqref{with-kernel} by a collection of finitely overlapping rectangular boxes $\tilde{\mathcal{U}}_\tau:=\{\tilde{U}\}$, where each $\tilde U$ is a translation copy of $\tilde{U}_\tau$. } { By \eqref{essentially-in-U-tau} and Lemma \ref{wandering-around}, it follows that, when $(x,t)$ is fixed and $(u_1, t_1)$ lies in some  $\tilde{U}\in \tilde{\mathcal{U}}_\tau$, we only need to consider $(u_2,t_2)$ contained in some translation of $\tilde{U}_\tau$, denoted by $\tilde{U}_{x,t}$. We would like to point out that, 
 \begin{enumerate}[(i)]
     \item $\tilde{U}_{x,t}$ only depends on $\tilde{U}$, not on specific point $(u_1, t_1)\in\tilde{U}$.
     \item As $\tilde{U}$ runs over all rectangular boxes parallel to $\tilde{U}_\tau$, so does $\tilde{U}_{x,t}$.
 \end{enumerate}
   Therefore, with the estimate \eqref{upper-bound-K-theta}, we have \eqref{with-kernel}
	\begin{align*}&\lesssim R^{-3}\int_{B_R}\sum_\sigma\sum_{\tau}\sum_{\tilde{U}\in \tilde{\mathcal{U}}_\tau}\int_{\tilde{U}}\left|\sum_{\theta\in\mathbf S_{R^{-1/2}}:\,\theta\subset\tau}\int_{\tilde{U}_{x,t}}|\tilde{U}|^{-1}|\tilde{F}_\theta(u_2,t_2)|^2\,du_2\,dt_2\right|^2\,du_1\,dt_1\,dx\,dt\\&\lesssim R^{-3}\int_{B_R} \sum_\sigma\sum_{\tau}\sum_{\tilde{U}\in \tilde{\mathcal{U}}_\tau}|\tilde{U}|^{-1}\left(\int_{\tilde{U}_{x,t}}\sum_{\theta\in\mathbf S_{R^{-1/2}}:\,\theta\subset\tau}|\tilde{F}_\theta|^2\right)^2\,dx\,dt\\&= R^{-3}\int_{B_R} \sum_\sigma\sum_{\tau}\sum_{\tilde{U}\in \tilde{\mathcal{U}}_\tau}|\tilde{U}|^{-1}\left(\int_{\tilde{U}}\sum_{\theta\in\mathbf S_{R^{-1/2}}:\,\theta\subset\tau}|\tilde{F}_\theta|^2\right)^2\,dx\,dt\\&\approx \sum_\sigma\sum_{\tau}\sum_{\tilde{U}\in \tilde{\mathcal{U}}_\tau}|\tilde{U}|^{-1}\left(\int_{\tilde{U}}\sum_{\theta\in\mathbf S_{R^{-1/2}}:\,\theta\subset\tau}|\tilde{F}_\theta|^2\right)^2,\end{align*}
 where the second line follows from (i), the third line follows from (ii), and the last line follows because nothing in the integrand depends on $(x,t)$. 
 
 Finally, inside the square, we change variables back from $\tilde{F}_\theta(u,t)$ to $F_\theta(x,t)$. Since $|\tilde{U}|\approx|U|$, this quantity is approximately equal to
 $$\sum_\sigma\sum_{\tau}\sum_{U\in\mathcal U_\tau}|U|^{-1}\left(\int_{U}\sum_{\theta\in\mathbf S_{R^{-1/2}}:\,\theta\subset\tau}|F_\theta|^2\right)^2,$$
	as desired.}

	\section{Lorentz Rescaling} \label{S lorentz}

	In this section, we establish the Lorentz rescaling lemma, which plays a crucial role in the induction on scales argument. Lorentz rescaling for the cone is the analogous counterpart of parabolic rescaling for the parabola, and both of them connect estimates from different scales. To proceed, we first  formulate an equivalent yet more quantitative definition of wave envelopes.
	
	\subsection{An equivalent description of $U$}
	Let $0<s\leq  1$ be a dyadic number and fix an arc $\tau\subset S^1$ with center $\eta_\tau\in S^1$ and aperture ${\rm d}(\tau)=s$.
	There exists an orthogonal matrix $A_\tau$ such that $A_\tau {\bf e_2}=\eta_\tau$. For each $m=(m_1,m_2)\in \mathbb{Z}^2$, let $m_\tau=A_\tau(s^{-1}m_1,m_2)\in\R^2$. We define $B_{\tau,m}$ by
	\beqq
	B_{\tau,m}:=\{y=(y_1,y_2)\in \R^2: |\Pi_{\eta_\tau}(y-m_\tau)|\leq C, |\Pi_{\eta_\tau^{\perp}}(y- m_\tau)|\leq Cs^{-1}\}.
	\eeqq
	where $\Pi_{\eta_\tau},\,  \Pi_{\eta_\tau^{\perp}}$  denote the projection onto $\eta_\tau$ and $\eta_\tau^{\perp}$ respectively.
	It is then straightforward to check that the whole space $\mathbb{R}^2$ is covered by $\{B_{\tau,m}\}_m$ with finite overlaps. Define $U_{\tau,m}$ to be
	\beq\label{eq:wave}
 \begin{aligned}
	\{z: |\Pi_{\eta_\tau}(\partial_\eta\phi^\lambda(z,\eta_\tau)-Rs^2m_\tau)|&\leq CRs^2,\\& |\Pi_{\eta_\tau^{\perp}}(\partial_\eta\phi^\lambda(z,\eta_\tau)-Rs^2m_\tau)|\leq CRs,|t|\leq R\}.
 \end{aligned}
	\eeq
	One may check that $U_{\tau,m}$ gives a curved slab of length $R$ and a rectangular cross-section of size $Rs^2\times Rs$.
	
	We claim that the definition of the wave envelop in \eqref{eq:wave} is equivalent to that in \eqref{eq:w1}. Indeed, by  changing variables
	$$x\mapsto x^\lambda_\tau(u,t), t\mapsto t,$$
	the curved tube $U_{\tau,0}$ is mapped to
	\beqq
	\{(u,t):|t|\leq R, |\Pi_{\eta_\tau}(u)|\leq CRs^2, |\Pi_{\eta_\tau^{\perp}}(u)|\leq CRs\},
	\eeqq
	which corresponds to $V_\tau\times (-R,R)$. Each $U_{\tau,m}$ then corresponds to some $V \in \mathcal{V}_\tau$. Therefore our claim is true.

	We cover $B_R$ by $\{U_{\tau,m}\}_m$ so that
	$$B_R\subset \bigcup_m U_{\tau,m}, $$
	with finite overlaps.

	We consider $S(r,R,\lambda)$ as defined in \eqref{def-S}. Using the equivalent definition of the curved tube $U_\tau$, we  have, up to a rapid decay in $\lambda$,
	\begin{equation*}
	\begin{aligned}
	&\sum_{B_r\subset B_R}|B_r|^{-1}\Big\|\Big(\sum_{\theta\in \mathbf{S}_{r^{-1/2}}}|\mathcal{F}^\lambda_\theta f|^2\Big)^{1/2}\Big\|_{L^2(B_r)}^4\\\leq \ & S(r,R,\lambda)\sum_{ R^{-1/2}\leq s\leq 1}\sum_{\tau\in\mathbf S_s}\sum_{m} |U_{\tau,m}|^{-1}\Big\|\Big(\sum_{\theta \subset \tau}|\mathcal{F}^\lambda_\theta f|^2\Big)^{1/2}\Big\|_{L^2(U_{\tau,m})}^4.
	\end{aligned}
	\end{equation*}

	\subsection{Lorentz Rescaling}
	
	With the above ingredients, we are ready to establish the Lorentz rescaling lemma. First, we take $\phi(x,t,\eta)=x\cdot \eta+t|\eta|$ as an example to illustrate the idea. In this case the hypersurface associated with $x\cdot \eta+t|\eta|$ is the standard cone $(\eta, |\eta|)$. To proceed, we apply the (Euclidean) Lorentz transformation to the cone first, which will put us in a better position to perform a parabolic rescaling. 
	By the change of variables
	
	\begin{equation}
	\begin{aligned}
	\eta_1\mapsto \eta_1,\quad  \eta_2\mapsto \frac{1}{\sqrt{2}}(\eta_3-\eta_2),\quad \eta_3\mapsto \frac{1}{\sqrt{2}}(\eta_3+\eta_2),
	\end{aligned}
	\end{equation}
	the cone $(\eta,|\eta|)$ is transformed to the \textit{tilted} cone $\eta_3=\eta_1^2/2\eta_2,$ the formula of which is invariant under parabolic rescaling  in $\eta_2$ and $\eta_3$. Moreover, this cone is tangent to the $\eta_1\eta_2$-plane.
	See Figure \ref{fig2}.
	
	\begin{figure}
		\centering
		\includegraphics[width=.60\textwidth]{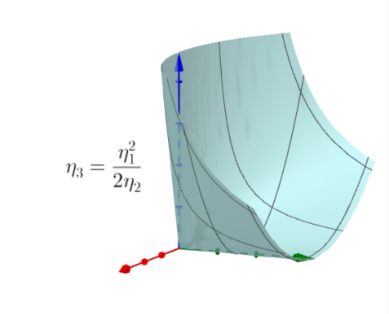}
		\caption{The cone after Lorentz transformation.}
		\label{fig2}
	\end{figure}
	
	For our case, the phase function $\phi(x,t,\eta)$  satisfies the cinematic curvature condition, which may not be translation invariant.  A significant difference is that 
	the associated cone in phase space is given by $(z,\partial_z\phi(z,\eta))$, which varies smoothly depending  on $z$. 
	Therefore the Euclidean Lorentz transformation does not work directly.
	To overcome this difficulty, we perform a change of variables with respect to the $(x,t)$-variables as displayed in Section \ref{changev}. 
	
	Denote \[\tilde{\phi}(u,t,\eta):=\phi(x_\tau(u,t),t,\eta).\]
	One can easily check that $\tilde{\phi}(u,t,\eta)$ also satisfies the cinematic curvature condition. By the discussion in Section \ref{changev}, the normal direction of the cone $ \partial_{u,t}\tilde{\phi}(u,t,\eta)$ is $(0,0,1)$ at $\eta=\eta_\tau$.  This manifests that the above change of variable plays a similar role as the Euclidean Lorentz transformation, since both of which tilt the cone so that it is tangent to the horizontal plane.  
	
	Now we restate the Lorentz rescaling lemma as follows:
	
	\begin{lemma}
		Let $r_1\leq r_2\leq r_3\leq R\leq \lambda^{1-\varepsilon}$, then
		$$S(r_1,r_3,\lambda)\leq \log r_2\, S(r_1,r_2,\lambda)\max_{r_2^{-1/2}\leq s\leq 1}S(s^2r_2,s^2r_3,s^2\lambda).$$
	\end{lemma}
	\begin{proof}
		By the definition of $S(r_1, r_2,\lambda)$, we obtain
		\begin{align}\label{eq-Lo0}
		\begin{aligned}&\sum_{B_{r_1\subset B_R}}|B_{r_1}|^{-1}\Big\|\Big(\sum_{\theta\in \mathbf{S}_{r_1^{-1/2}}}|\mathcal F^\lambda_\theta f|^2\Big)^{1/2}\Big\|_{L^2(B_{r_1})}^4\\\leq &  S(r_1,r_2,\lambda)\sum_{r_2^{-1/2}\leq s\leq 1}\sum_{\tau\in\mathbf S_s}\sum_{m}|U_{\tau,m}|^{-1}\Big\|\Big(\sum_{\theta \subset \tau:\theta\in \mathbf{S}_{r_2^{-1/2}}}|\mathcal F^\lambda_\theta f|^2\Big)^{1/2}\Big\|_{L^2(U_{\tau,m})}^4\\&+{\rm RapDec}(\lambda)\|f\|_{L^2}^4.
		\end{aligned}\end{align}
		To complete the proof, it suffices to show that for a fixed $\tau$ with aperture $s$, we have
		\begin{align}\label{eq-Lo1}\begin{aligned}& \sum_{m}|U_{\tau,m}|^{-1}\Big\|\Big(\sum_{\theta\subset \tau:\,\theta\in\mathbf S_{r_2^{-1/2}}}|\mathcal F^\lambda_\theta f|^2\Big)^{1/2}\Big\|_{L^2(U_{\tau,m})}^4\\
		\leq & S(s^2r_2,s^2r_3,s^2\lambda)\sum_{r_3^{-1/2}\leq s'\leq s}\sum_{\substack{\tau'\subset \tau\\\tau'\in \mathbf{S}_{s' }}}\sum_{m}|U_{\tau',m}|^{-1}\Big\|\Big(\sum_{\substack{\theta\subset \tau'\\ \theta\in \mathbf{S}_{r_3^{-1/2}}}}|\mathcal F^\lambda_\theta f|^2\Big)^{1/2}\Big\|_{L^2(U_{\tau',m})}^4\\&+{\ra}(\lambda)\|f\|_{L^2}^4.
		\end{aligned}\end{align}
		In fact, if we plug \eqref{eq-Lo1} into \eqref{eq-Lo0} and notice that the number of dyadic number $s$ located in $(r_2^{-1/2},1)$ is $\log r_2$, the lemma follows.
		
		Now we prove \eqref{eq-Lo1}. For a given $\tau$,  let  $A_\tau$ be  an orthogonal matrix such that $A_\tau {\bf e_2} =\eta_\tau$.  We shall perform the following change of variables
		$$x\mapsto x_\tau^\lambda (A_\tau u,t),\ t\mapsto t, \ \eta \mapsto  A_\tau \eta.$$
		Define
		\begin{align*}
		\tilde{\F}^\lambda_\theta f(u,t):&=\int e^{i\tilde{\phi}^\lambda (u,t,\eta)}\tilde{a}^\lambda_\theta (u,t,\eta)\hat{f}(\eta)\,d\eta,\\
		\tilde{\phi}^\lambda(u,t,\eta):&=\phi^\lambda(x_\tau^\lambda(A_\tau u,t),t,A_\tau\eta), \tilde{a}_{\theta}^\lambda(u,t,\eta)=a_\theta^\lambda( x_\tau^\lambda(A_\tau u,t),t, A_\tau \eta).
		\end{align*}
		Under the new  coordinates the curved tube $U_{\tau, m}$ is straightened and given by 
		\beq\label{eq:ts}
		\tilde{U}_{\tau, m}:=\{(u,t):|t|\leq r_2, |u_2-r_2s^2m_2|\leq r_2s^2, |u_1-r_2sm_1|\leq r_2s\},
		\eeq
		and the left hand side of \eqref{eq-Lo1} is transformed to
		\beq\label{eq:a2}
		\sum_{m} |\tilde{U}_{\tau,m}|^{-1}\Big\|\Big(\sum_{\theta\subset \tau:\theta\in \mathbf{S}_{r_2^{-1/2}}}|\tilde{\F}^\lambda_\theta f|^2\Big)^{\f{1}{2}}\Big)\Big\|_{L^2(\tilde{U}_{\tau,m})}^4.
		\eeq
		
		Next, we perform parabolic rescaling
		$$\eta_1\mapsto s\eta_1,\ \eta_2\mapsto \eta_2,\  u_1\mapsto s^{-1}u_1, \ u_2\mapsto u_2,\  t\mapsto s^{-2}t.$$
		Define $\tilde\phi_s$  and $\bar a$ to be
		\begin{equation*}
		\begin{aligned}
		\tilde\phi_s(u,t,\eta):=s^{-2}\tilde{\phi}(su_1,s^2u_2,t,s\eta_1,\eta_2),\\
		\bar a_\theta(u,t,\eta):=\tilde{a}_\theta (su_1,s^2u_2,t,s\eta_1,\eta_2).
		\end{aligned}
		\end{equation*}

		Correspondingly, by \eqref{eq:ts}, $\tilde{U}_{\tau,m}$ is transformed to a cube $Q_m$ of side-length $ r_2s^2$, namely
		\beqq
		Q_m=\{(u,t), |t|\leq r_2s^2:|u_2-r_2s^2m_2|\leq r_2s^2, |u_1-r_2s^2m_1|\leq r_2s^2 \}.
		\eeqq
		Then \eqref{eq:a2} becomes
		\beqq
		\sum_{m}|Q_m|^{-1}\Big\|\Big(\sum_{\theta\subset \tau:\,\theta\in \mathbf{S}_{r_2^{-1/2}}}|\bar {\F}^{\lambda s^2}_\theta  \bar f|^2\Big)^{\f{1}{2}}\Big)\Big\|_{L^2(Q_m)}^4,
		\eeqq
		where $\bar{\F}^{\lambda s^2}_\theta \bar f$ is defined by 
		
		$$\bar{\F}^{\lambda s^2}_\theta\bar f(u,t):=\int e^{i\tilde\phi_s^{\lambda s^2}(u,t,\eta)}\bar{a}^{\lambda s^2}_\theta (u,t,\eta)\hat{\bar f}(\eta)\,d\eta \quad \text{with}\;\; \hat{\bar f}(\eta):=\hat{f}(s\eta_1,\eta_2).$$

		We claim that the phase function $\tilde\phi_s$ satisfies our assumptions on the phase functions. Assuming this, by the definition of $S(s^2r_2, s^2 r_3, s^2\lambda)$, we have,
		\begin{multline}
		\sum_{m}|Q_m|^{-1}\Big\|\Big(\sum_{\theta\subset \tau:\theta\in \mathbf{S}_{r_2^{-1/2}}}|\bar {\F}^{\lambda s^2}_\theta  \bar f|^2\Big)^{\f{1}{2}}\Big)\Big\|_{L^2(Q_m)}^4\\
		\leq S(s^2r_2, s^2 r_3, s^2\lambda)\sum_{r_3^{-1/2} \leq s'\leq s}\sum_{\substack{\tau'\subset \tau\\\tau'\in \mathbf{S}_{s'}}}\sum_m |U_{\tau',m}|^{-1}\cdot\\\Big\|\Big(\sum_{\theta'\subset \tau':\theta'\in \mathbf{S}_{r_3^{-1/2}}}|\bar{\F}^{\lambda s^2}_{\theta'}  \bar f|^2\Big)^{\f{1}{2}}\Big)\Big\|_{L^2(U_{\tau',m})}^4 +{\rm RapDec}(\lambda)\|f\|_{L^2}^4 .
\end{multline}
		Note that  $U_{\tau',m}$ is defined by
		\begin{multline*}
		U_{\tau',m}:=\{(u,t),|t|\leq r_3s^2: |\Pi_{\eta_{\bar{\tau}'}}(\partial_\eta \tilde\phi_s^{\lambda s^2}(u,t,\eta_{\bar{\tau}'})-r_3s^2m_{\bar \tau'})|\leq r_3(s')^2, \\ |\Pi_{\eta_{\bar{\tau}'}^{\perp}}(\partial_\eta \tilde\phi_s^{\lambda s^2}(u,t,\eta_{\bar{\tau}'})-r_3s^2m_{\bar \tau'})|\leq r_3 s s'\},
		\end{multline*}
		with  $\eta_{\bar{\tau}'} \in S^1$ satisfying
		$$\eta_{{\tau}'}\parallel A_\tau  \begin{pmatrix}s & 0\\0&1\end{pmatrix} \eta_{\bar{\tau}'},$$ 
		and $m_{\bar\tau'}:=A_{\bar \tau'}(s/s'm_1,m_2)$. After changing back to the original variables, we obtain  \eqref{eq-Lo1}.

		To conclude this section, it remains to check that $\tilde\phi_s$ satisfies our assumptions on the phase  and the amplitude functions, that is, $\tilde\phi_s$ satisfies the quantitative conditions ${\bf {\bf C_2}}$ and  the cinematic curvature  condition and the conditions  ${\bf {\bf C_1}}$, ${\bf {\bf C_3}}$ on the amplitude functions.
		
		It is easy to verify the conditions ${\bf {\bf C_1}}$, ${\bf {\bf C_3}}$. Indeed, one can employ the following change of variables
		\beqq
		(u,t)\mapsto (u,t)/A, \eta\mapsto \eta/A,
		\eeqq
		where $A>0$ is a sufficiently large absolute constant, and then decompose the enlarged domain into a finite number of smaller ones. Thus, it suffices to verify the condition ${\bf {\bf C_2}}$.

		For the new phase function $\tilde \phi_s$, by homogeneity, we have   \beq\label{change3}
		s^{-2}\tilde{\phi}(su_1,s^2 u_2,t,s\eta_1,\eta_2)=s^{-2}\eta_2\tilde{\phi}(su_1,s^2 u_2,t,s\eta_1/\eta_2,1).
		\eeq
		Using Taylor's formula, we get 
		\beq\label{eq:change2}
		\begin{aligned}
			\tilde{\phi}(su_1,s^2 u_2,t,s\eta_1/\eta_2,1)&=\tilde{\phi}(su_1,s^2 u_2,t,{\bf e_2})+s\partial_{\eta_1}\tilde{\phi}(su_1,s^2 u_2,t,{\bf e_2})\eta_1/\eta_2\\&
			+1/2s^2\partial_{\eta_1}^2\tilde{\phi}(su_1,s^2 u_2,t,{\bf e_2})(\eta_1/\eta_2)^2+O(s^3(\eta_1/\eta_2)^3).
		\end{aligned}
		\eeq
		By Euler's formula, we obtain 
		\beq\label{eq:change1}
		\tilde{\phi}(su_1,s^2 u_2,t,{\bf e_2})=\partial_{\eta_2}\tilde{\phi}(su_1,s^2 u_2,t,{\bf e_2}).
		\eeq
		Taking \eqref{eq:change1} and \eqref{eq:change2} into
		\eqref{change3} and invoking \eqref{change-to-u}, we have 
		\beqq
		\tilde{\phi}_s=s^{-2}\tilde{\phi}(su_1,s^2 u_2,t,s\eta_1,\eta_2)=u\cdot \eta+\frac{1}{2}\partial_{\eta_1}^2\tilde{\phi}(su_1,s^2 u_2,t,{\bf e_2})\eta_1^2/\eta_2+O(s(\eta_1^3/\eta_2^2)).
		\eeqq
		where the implicit constant in the big $O$ only depends on a finite number of derivatives of $\phi$.  Therefore, the condition ${\bf {\bf C_2}}$ can be ensured. 
		
		Finally, let us check that $\tilde\phi_s$ satisfies the cinematic curvature condition. 
		\begin{equation*}
		\tilde\phi_s(u,t,\eta)=s^{-2}\phi(x_\tau(A_\tau(su_1,s^2u_2),t),t,A_\tau(s\eta_1,\eta_2)).
		\end{equation*}
		By the homogeneity of $\phi$, it is clear that $\tilde\phi_s(u,t,\eta)$ is homogeneous of degree $1$ with respect to $\eta$. Since $A_\tau {\bf e_2}=\eta_\tau$, we have
		\begin{equation}\label{eq:lo12}
		\begin{aligned}
		s^{-2}\partial_\eta \tilde\phi_s(u,t,{\bf e_2})&=s^{-2}\begin{pmatrix}s & 0\\0&1\end{pmatrix}A_\tau^{t}\partial_\eta\phi(x_\tau(A_\tau(su_1,s^2u_2),t),t,\eta_\tau)\\
		&=s^{-2}\begin{pmatrix}s & 0\\0&1\end{pmatrix}A_\tau^t A_\tau \begin{pmatrix}s & 0\\0&s^2\end{pmatrix}\begin{pmatrix}u_1 \\u_2 \end{pmatrix} =\begin{pmatrix}u_1 \\u_2 \end{pmatrix}.
		\end{aligned}
		\end{equation}
		Therefore
		$$s^{-2}\partial_{\eta}\partial_u\tilde\phi_s(u,t,{\bf e_2})=\begin{pmatrix}1 & 0\\0&1\end{pmatrix},$$
		and then the non-degeneracy condition \eqref{CC1} follows from the continuity of the phase function and our assumption of smallness on the support of $a(x,t,\eta)$.

		Now we verify the curvature condition \eqref{CC2} of the phase function $\tilde\phi_s$. The two tangent vectors of the  hypersurface $\{\partial_{(u,t)}\tilde\phi_s(u,t,\eta):\eta\in {\rm supp}\, a^\lambda(z,{}\cdot{})\}$ at ${\bf e_2}$ are $$\partial_{\eta_1}\partial_{(u,t)}\tilde\phi_s(u,t,{\bf e_2}),\ \partial_{\eta_2}\partial_{(u,t)} \tilde\phi_s(u,t,{\bf e_2}).$$ 
		
		Recall from \eqref{eq:lo12} that we have normalized the normal vector to the hypersurface $\{\partial_{(u,t)}\tilde\phi_s(u,t,\eta):\eta\in {\rm supp}\, a^\lambda(z,{}\cdot{})\}$ at ${\bf e_2}$ to be ${\bf e_3}$.  A simple  calculation gives that
		\begin{equation}\label{eq-Lo3}
		\partial_{\eta\eta}^2\langle \partial_{(u,t)}\tilde\phi_s(u,t,\eta),{\bf e_3}\rangle|_{\eta={\bf e_2}}=\partial_{\eta\eta}^2\partial_t\tilde\phi_s(u,t,\eta)|_{\eta={\bf e_2}}.
		\end{equation}
  Since 
  \beqq
\partial_{\eta\eta}^2\partial_t\tilde{\phi}_s(u,t,\eta)|_{\eta={\bf e_2}}=B^{t}\partial_{\eta\eta}^2\partial_t\phi(x_\tau(A_\tau(su_1,s^2u_2),t),t,\eta_\tau) B,
  \eeqq
  where 
  $$B=A_\tau \begin{pmatrix}s & 0\\0&1\end{pmatrix}.$$
		
		Also, note that 
		\beq\label{eq:nond}
		{\rm rank}\,\partial_{\eta\eta}^2\langle\partial_{(u,t)}\phi(x_\tau(u,t),t,\eta_\tau),{\bf e_3}\rangle={\rm rank}\,\partial_{\eta\eta}^2\partial_t\phi(x_\tau(u,t),t,\eta_\tau)=1.
		\eeq
		Finally, by \eqref{eq:nond},\eqref{eq-Lo3} and the non-degeneracy of $B$, we have $${\rm rank}\,\partial_{\eta\eta}^2\partial_t\tilde\phi_s(u,t,\eta)|_{\eta={\bf e_2}}=1.$$ The proof of the claim is complete.
	\end{proof}

	\section{Appendix}

	\subsection{Theorem \ref{sq} implies local smoothing estimate.} For the reader's convenience, we provide 
	details on how to get a local smoothing estimate from  Theorem \ref{sq} via an argument from \cite{MSSJ}. 
	
	Let $\theta \subset S^1$ be an arc of aperture $ \lambda^{-1/2}$. By a standard Littlewood-Paley decomposition, Conjecture \ref{Conj. 3} can be deduced from \eqref{eq:34}, which says
	\begin{align*}
	\|\mathcal{F}^\lambda f\|_{L^{ 4}(\R^{3})}\leq C_\varepsilon \lambda^{1/4+\varepsilon}\|f\|_{L^{4}(\R^2)}.
	\end{align*}

	Assuming the square function estimate
	\beqq
	\Big\|\sum_{\theta} \F_\theta^\lambda  f \Big\|_{L^{ 4}(\mathbb{R}^{3})}\leq_\varepsilon \lambda^{\varepsilon} \Big\|\Big(\sum_\theta |\F^\lambda_\theta f|^2\Big)^{1/2}\Big\|_{L^{4}(\R^{3})}, \eeqq
	one can see that to establish \eqref{eq:Conj. 3} it suffices to prove
	\beqq
	\Big\|\Big(\sum_\theta |\F^\lambda_\theta f|^2\Big)^{1/2}\Big\|_{L^{ 4}(\R^{3})}\leq_\varepsilon \lambda^{1/4+\varepsilon}\|f\|_{L^{ 4}(\R^2)}.
	\eeqq
	
	To this end,  we use $K_\theta^\lambda$ to denote the kernel of the operator $\F^\lambda_\theta $, i.e. 
	\beqq
	K_\theta^\lambda (x,t,y):=\int_{\R^2} e^{i(\phi^\lambda(x,t,\eta)-y\cdot\eta)} \chi_\theta(\eta)\,d\eta.
	\eeqq
	By integration by parts in $\eta$, we have for each $N\in \mathbb N$ 
	\beq \label{eq:Niko}
	|K_\theta^\lambda  (x,t,y)|\lesssim_N\lambda^{-1/2}  \big(1+|\Pi_{\eta_\theta}(\partial_\eta\phi^\lambda(z,\eta_\theta)-y)|+\lambda^{-1/2}|\Pi_{\eta_\theta^{\perp}}(\partial_\eta\phi^\lambda(z,\eta_\theta)-y)|\big)^{-N}.
	\eeq
	
	By  Riesz's representation theorem, we have
	\beqq
	\Big\|\Big(\sum_\theta |\F^\lambda_\theta f|^2\Big)^{1/2}\Big\|_{L^{4}(\R^{3})}^2=\sup_{\|g\|_{L_{x,t}^2(\R^{3})=1}}\int_{\R^{3}} \sum_{\theta}|\F^\lambda_\theta f(x,t)|^2 g(x,t) \,dx\,dt.
	\eeqq
	Then, by H\"older's inequality
	\begin{align*}
	\quad &\int_{\R^{3}} \sum_{\theta }|\F^\lambda_\theta f(x,t)|^2 g(x,t) \,dx\,dt\\
	&=\int_{\R^3} \sum_\theta \Big|\int_{\R^2} K_\theta^\lambda (x,t,y)f_\theta(y)\,dy\Big|^2 g(x,t)\,dx\,dt\\
	&\lesssim \int  \sum_\theta\left( \int_{\R^2}|K_\theta^\lambda(x,t,y)|\,dy\right)\left( \int |K_\theta^\lambda(x,t,y)| |f_\theta(y)|^2 \,dy\right) |g(x,t)|\,dx\,dt\\
	&\lesssim  \left(\sup_{\theta,x,t}\int |K_\theta^\lambda(x,t,y)|\,dy\right)\int_{\R^2}\Big(\sup_\theta\int |K_\theta^\lambda(x,t,y)||g(x,t)|\,dx\,dt\Big)\sum_\theta |f_\theta|^2 \,dy\\
	&\lesssim \int_{\R^2} \left(\sum_{\theta} |f_\theta(y)|^2\right)\left( \sup_{\theta}\int_{\R^{3}}|K_\theta^\lambda (x,t,y)g(x,t)|\,dx\,dt\right)\,dy\\
	&\lesssim \Big\|\sum_{\theta}|f_\theta|^2\Big\|_{L^2(\R^2)}\Big\|\sup_{\theta}\int_{\R^{3}}|K_\theta^\lambda  (x,t,y)g(x,t)|\,dx\,dt\Big\|_{L^2(\R^2)},
	\end{align*}
	where in the second to last line, we use \eqref{eq:Niko} to obtain
	\beqq
	\sup_{\theta,x,t}\int |K_\theta^\lambda(x,t,y)|\,dy\lesssim 1.
	\eeqq

	Then the following two inequalities complete the proof.
	
	\noindent{\bf Square function inequality }
	\beq\label{eqwww}
	\Big\| \Big(\sum_{\theta}|f_\theta|^2\Big)^{\f{1}{2}}\Big\|_{L^p(\R^2)}\lesssim \lambda^\varepsilon \|f\|_{L^p(\R^2)},\;\;\; 2\leq p\leq 4.
	\eeq
	
	\noindent {\bf  Nikodym maximal function }
	\beq \label{w10}
	\Big\|\sup_{\theta}\int_{\R^{3}}|K_\theta^\lambda (x,t,y)g(x,t)|\,dx\,dt\Big\|_{L^2(\R^2)}\lesssim_\varepsilon \lambda^{1/2+\varepsilon}.
	\eeq
	
	\eqref{eqwww} can be found in \cite{Cor}. It remains to check \eqref{w10}. One can see from \eqref{eq:Niko} that, when $\theta, y$ are fixed, the essential support of $K_\theta^\lambda$ is contained in a curved tube $T_{\theta,y}$ defined by 
	\beqq
	T_{\theta,y}:=\{(x,t):|t|\leq \lambda , |\Pi_{\eta_\theta}(\partial_\eta\phi^\lambda(x,t,\eta_\theta)-y)|\leq 1, |\Pi_{\eta_\theta^{\perp}}(\partial_\eta\phi^\lambda(x,t,\eta_\theta)-y)|\leq \lambda^{1/2}\}.
	\eeqq
	Thus, we may view \eqref{w10} as a Nikodym-type maximal function estimate, in the sense that 
	\beq \label{w12a}
	\Big(\int_{\R^2}\sup_\theta \Big|\f{1}{|T_{\theta,y}|}\int_{T_{\theta,y}}|g(x,t)|\,dx\,dt\Big|^2 \,dy\Big)^{\f{1}{2}}\lesssim_\varepsilon  \lambda^{-1/2+\varepsilon }\|g\|_{L^2(\R^3)}.\eeq
	\eqref{w12a}, in turn,  is a direct consequence of the following  Nikodym maximal function estimate from \cite{MSSJ}.

	\begin{theorem}\cite{MSSJ}
		\label{conN}
		Assume $1\leq R\leq \lambda$ and $\phi$ satisfies the \it{cinematic curvature} condition.     Let  tube ~$\mathcal{T}_{\theta,y}$ be defined as follows
		\begin{align*}
		\mathcal{T}_{\theta,y}:=\big\{(x,t)\in B_R(0) \times[0,R]: |(\partial_\eta\phi^\lambda(x,t,\eta_\theta)-y)|\leq
		1\big\},
		\end{align*}
		then
		\beq \label{w12}
		\Big(\int_{\R^2}\sup_\theta \Big|\f{1}{|\mathcal{T}_{\theta,y}|}\int_{\mathcal{T}_{\theta,y}}|g(x,t)|\,dx\,dt\Big|^2 \,dy\Big)^{\f{1}{2}}\lesssim_\varepsilon  R^{-1/2+\varepsilon }\|g\|_{L^2(\R^3)}.
		\eeq
	\end{theorem}
	
	It is not hard to see that \eqref{w12} implies \eqref{w12a}. Indeed, one may decompose $T_{\theta,y}$ into a collection of finitely-overlapping tubes $\{\mathcal{T}_{\theta, y'}\}_{y'}$ such that 
	$$T_{\theta,y}\subset \bigcup\limits_{y'}\mathcal{T}_{\theta,y'}.$$
	Therefore,
	
	\beq
	\f{1}{|T_{\theta,y}|}\int_{T_{\theta,y}}|g(x,t)|\,dx\,dt\leq {\rm Avg}_{y'} \f{1}{|\mathcal{T}_{\theta,y'}|}\int_{\mathcal{T}_{\theta,y'}}|g(x,t)|\,dx\,dt
	\eeq
	where  ${\rm Avg}_{n}$  is defined by 
	$$ {\rm Avg}_{n}a_n:= \frac{1}{ \#\{n\}} \sum_n a_n.$$
	
	By Minkowski's inequality, \eqref{w12a} follows from \eqref{w12}. 
	
	Finally, we remark that the higher dimensional analog of the above Nikodym maximal function estimate \eqref{w12} is an interesting open problem. Nevertheless, as was shown in \cite{BHS} (\cite{MS} for manifolds), in higher dimensions, we, in general, do not have favorable variable coefficient Nikodym maximal function estimates, and therefore going from square function estimates to local smoothing bounds is more subtle in higher dimensions.
	{ \subsection{Finite overlapping of the planks for a general cone}
Let $f:[0,1]\rightarrow \mathbb{R}$ be 	a smooth function with $|f'|\leq C_1$ and $9/10\leq f''\leq 11/10$, where $C_1>0$ is an absolute constant.  We shall use $f$ to model (a local piece of) a general cone with its curvature bounded above and below. Define the cone $$\mathcal{C}(s,t):=\{(st, sf(t),s): 0\leq s\leq 1 , 0\leq t\leq 1\}.$$
Let \begin{equation*}
	\mathbf{C}(t):=\frac{(t,f(t),1)}{\sqrt{1+t^2+f(t)^2}},\quad \mathbf{T}(t):=\frac{(1,f'(t),0)}{\sqrt{1+|f'(t)|^2}},\quad \mathbf{N}(t):=\frac{\mathbf{C}(t)\times \mathbf{N}(t)}{|\mathbf{C}(t)\times \mathbf{N}(t)|}.
	\end{equation*}
Let $r^{-1}<\sigma \leq 1$ and $0\leq i\leq \tfrac{1}{2}\sigma r$ with $i\in \mathbb{N}$. Define  $t_i:=(i+\tfrac{1}{2})\sigma^{-1}r^{-1},$.  Define the planks 
$$\mathbf{\Theta}(\sigma,t_i):=\{\xi \in \mathbb{R}^3: |\mathbf{C}(t_i)\cdot \xi|\leq \sigma^2, |\mathbf{T}(t_i)\cdot \xi|\leq r^{-1}\sigma, |\mathbf{N}(t_i)\cdot \xi|\leq r^{-2}\}.$$
Let 
$$\mathbf{\Theta}_\sigma:=\bigcup_{i}\mathbf{\Theta}(\sigma,t_i).$$

Now we are ready to state a generalization of \cite[Lemma 4.2] {GWZ} for cone $\mathcal C$.

\begin{lemma}\label{general cone}
	If $\xi \in \mathbf{\Theta}_\sigma \backslash \mathbf{\Theta}_{\sigma/2}$, then there exist at most $C_0$ planks which are contained in $\mathbf{\Theta}_{\sigma}$ overlapping at $\xi$, where $C_0$ is an absolute constant that only depends on $f$.
	\end{lemma}	
	\begin{proof}
It is easy to verify the case when $\xi $ is in  the annulus $$A_\sigma:=\{\xi: \frac{\sigma^2}{4}\leq |\xi|\leq \sigma^2\},$$ the planks $\{\mathbf{\Theta}(\sigma,t_i)\cap A_\sigma\}_i$ are finitely overlapping.  It remains to consider the case when $|\xi|\leq \frac{\sigma^2}{4}$.

If $|\xi|\leq \sigma^2/4$ and $\xi \in \mathbf{\Theta}_\sigma \backslash \mathbf{\Theta}_{\sigma/2}$, then $\xi \in \bigcup\limits_{i} \widetilde{\mathbf{\Theta}}(\sigma,t_i)$,
where $$\widetilde{\mathbf{\Theta}}(\sigma,t_i):=\{\xi \in \mathbb{R}^3: |\mathbf{C}(t_i)\cdot \xi|\leq \sigma^2/4, cr^{-1}\sigma\leq |\mathbf{T}(t_i)\cdot \xi|\leq r^{-1}\sigma, |\mathbf{N}(t_i)\cdot \xi|\leq r^{-2}\},$$
	and $c$ is a fixed small constant.  

 Fix a point $\xi=(\xi_1,\xi_2,s_0)\in\bigcup\limits_{i} \widetilde{\mathbf{\Theta}}(\sigma,t_i)$ on the plane $\xi_3=s_0\approx\sigma^2$. By the strict convexity of the curve $\mathcal{C}(s_0,t)$, we know that there are at most two lines $\ell_j,\ j=1,2$ that are tangent to $\mathcal{C}(s_0,t)$ at $\mathcal{C}(s_0,\tau^\xi_j)$ for some $\tau^\xi_j\in[0,1],\ j=1,2.$ Without loss of generality, by possibly extend the domain of $f$, we may assume that there are exactly two such lines. 
 
 It now suffices to show that if a line $\ell_k$ that is tangent to $\mathcal{C}(s_0,t)$ at $\mathcal{C}(s_0,t_{i_k})$ and it passes through $B_{r^{-2}}(\xi)$, then $|t_{i_k}-\tau^\xi_1|\leq C\sigma^{-1}r^{-1}$ or $|t_{i_k}-\tau^\xi_2|\leq C\sigma^{-1}r^{-1}$ for some absolute constant $C$. In fact, it is elementary to check that if $|\xi- \tilde\xi|\le r^{-2}$, then $|\tau^\xi_j-\tau^{\tilde\xi}_j|\le C \sigma^{-1}r^{-1}$ for an absolute constant $C$ that only depends on the cone. See Figure \ref{fig3}.
  \begin{figure}
		\centering
		\includegraphics[width=0.8\textwidth]{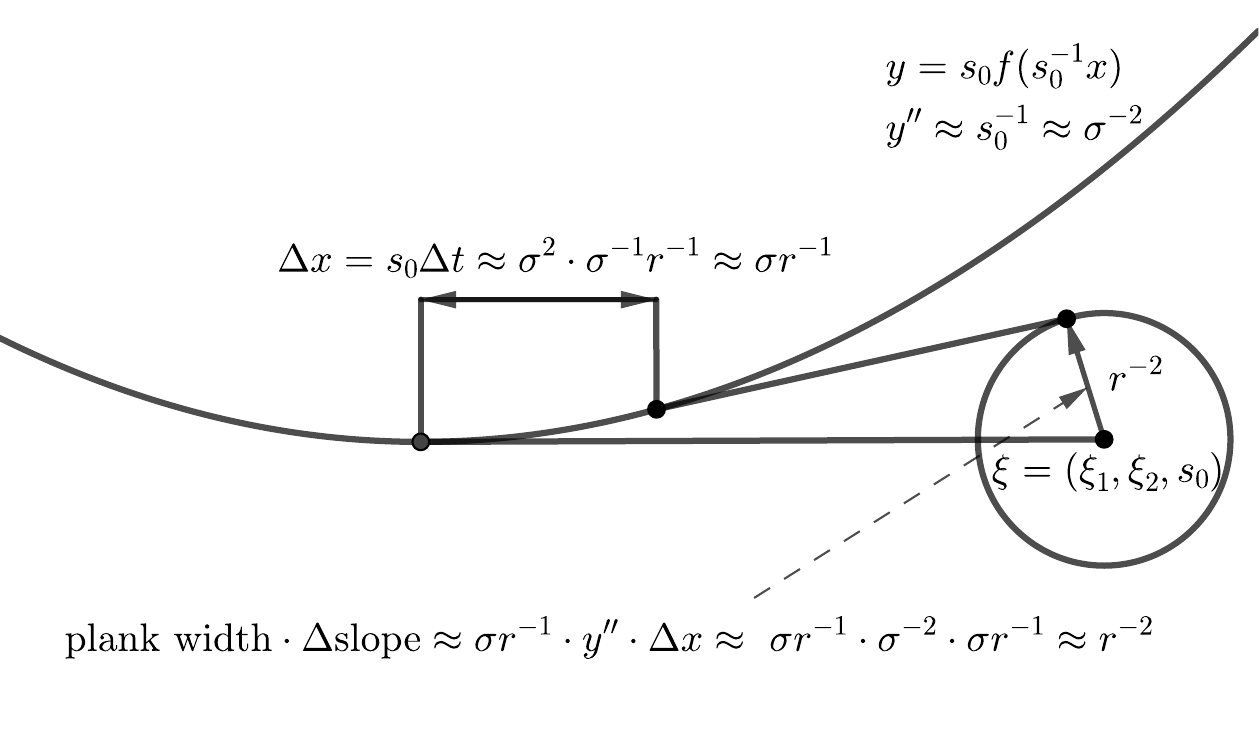}
		\caption{Tangent lines passing through $B_{r^{-2}}(\xi).$}
		\label{fig3}
	\end{figure}
	\end{proof}
	}

	\bibliography{F}{}
	\bibliographystyle{plain}

\end{document}